\documentclass[]{siamart171218}

\usepackage{amsfonts}
\usepackage{amssymb,mathtools,amsmath}
\usepackage{fixmath}
\usepackage{graphicx}
\usepackage{epstopdf}
\usepackage{algorithmic}
\usepackage{braket}
\usepackage{upgreek}
\usepackage{enumerate}

\usepackage{tabularx}
\usepackage{array,booktabs} 

\newtheorem{assumption}{Assumption}[section]
\newtheorem{remark}[theorem]{Remark}

\crefname{assumption}{Assumption}{Assumptions}
\crefname{proposition}{Proposition}{Propositions}
\crefname{section}{Section}{Sections}
\crefname{subsection}{Subsection}{Subsections}

\ifpdf
\DeclareGraphicsExtensions{.eps,.pdf,.png,.jpg}
\else
\DeclareGraphicsExtensions{.eps}
\fi

\numberwithin{theorem}{section}
\numberwithin{equation}{section}

\newcommand{\TheTitle}{Error estimates for space-time discretization of parabolic time-optimal control problems with bang-bang controls} 
\newcommand{\TheAuthors}{Lucas Bonifacius, Konstantin Pieper, and Boris Vexler}

\headers{Error Estimates Time-Optimal Control with Bang-bang Solutions}{\TheAuthors}

\title{{\TheTitle}%
\thanks{\textbf{Funding:}
The first author gratefully acknowledge support from the International Research
Training Group IGDK, funded by the German Science Foundation (DFG) and the Austrian
Science Fund (FWF).}}

\author{
	Lucas Bonifacius\thanks{Fakult\"at f\"ur Mathematik, Technische Universit\"at M\"unchen	(\email{lucas.bonifacius@tum.de}; \email{vexler@ma.tum.de}).}
	\and
	Konstantin Pieper\thanks{Department of Scientific Computing, Florida State University (\email{kpieper@fsu.edu}).}
	\and
	Boris Vexler\footnotemark[2]
}

\usepackage{amsopn}

\ifpdf
\hypersetup{
  pdftitle={\TheTitle},
  pdfauthor={\TheAuthors}
}
\fi

\crefname{assumption}{Assumption}{Assumptions}
\crefname{proposition}{Proposition}{Propositions}
\crefname{corollary}{Corollary}{Corollaries}

\newcommand{\algtol}[1]{\ensuremath{\varepsilon_{\text{#1}}}}

\newcommand{\N}{\mathbb{N}}
\newcommand{\R}{\mathbb{R}}
\newcommand{\Rplus}{\R_+}

\newcommand{\ldef}{\coloneqq}
\newcommand{\rdef}{\eqqcolon}
\newcommand{\e}{\mathrm{e}}

\newcommand{\parameterControlDim}{N_c}

\newcommand{\constraintSet}{\colon}

\newcommand{\embedding}{\hookrightarrow}

\DeclareMathOperator*{\Id}{Id}
\newcommand{\dom}[2]{\mathcal{D}_\ensuremath{{#1}(#2)}}
\newcommand{\semigroup}[1]{\ensuremath{\e^{#1}}}

\newcommand{\MPRSpace}[4][I]{\ensuremath{W^{1,#2}(#1;#3)\cap L^{#2}(#1;#4)}}
\newcommand{\MPRHilbert}[3]{\ensuremath{H^{1}(#1;#2)\cap L^{2}(#1;#3)}}

\DeclarePairedDelimiter{\abs}{\lvert}{\rvert}
\DeclarePairedDelimiter{\norm}{\lVert}{\lVert}
\DeclarePairedDelimiter{\pair}{\langle}{\rangle}
\DeclarePairedDelimiter{\inner}{(}{)}

\DeclareMathOperator*{\esssup}{ess\,sup}
\DeclareMathOperator*{\argmin}{arg\!\min}

\newcommand{\Q}{\ensuremath{Q}}
\newcommand{\Qad}{\ensuremath{Q}_{ad}}
\newcommand{\Qsigma}{\ensuremath{Q}_\sigma}
\newcommand{\Qsigmaad}{\ensuremath{Q}_{ad,\sigma}}
\newcommand{\Qkh}{\ensuremath{Q}_{kh}}

\newcommand{\Qh}{\ensuremath{Q}_{h}}
\newcommand{\ControlOp}{\ensuremath{B}}
\newcommand{\Lcontrol}[2][I]{\ensuremath{L^{#2}(#1\times\omega)}}
\newcommand{\LcontrolSpatial}[1][2]{\ensuremath{L^{#1}(\omega)}}

\newcommand{\Xkh}{\ensuremath{X_{k,h}}}
\newcommand{\B}{\mathrm{B}}

\newcommand{\ProjDiscControl}{\mathrm{I}_\sigma}
\newcommand{\ProjH}{\mathrm{\Pi}_{h}}
\newcommand{\ProjHconst}{\mathrm{\Pi}_{h,0}}
\newcommand{\ProjK}{\mathrm{\Pi}_{k}}

\newcommand{\D}[1]{\,\mathrm{d}\ensuremath{#1}}
\newcommand{\Lap}{\upDelta}

\begin{document}
	
\maketitle

\begin{abstract}
	In this paper a priori error estimates are derived for 
	full discretization (in space and time) 
	of time-optimal control problems. 
	Various convergence results
	for the optimal time and the control variable
	are proved under different assumptions.
	Especially the case of bang-bang controls is investigated. 
	Numerical examples are provided to illustrate the results.
\end{abstract}



\section{Introduction}

In this article, we consider time-optimal control problems subject to parabolic
partial differential equations. 
More precisely, we study the following model problem,
where $u$ denotes the
state, $q$ the control, and $T$ the terminal time:
\begin{equation}\label{P}\tag{\ensuremath{P}}
\mbox{Minimize~} T
\quad\text{subject to}\quad
\left\{
\begin{aligned}
T &> 0\mbox{,}\\
\partial_t u -\Lap u &= Bq\mbox{,} &\mbox{in~} (0,T)\times\Omega\mbox{,}\\
u &= 0\mbox{,} & \mbox{on~} (0,T)\times \partial\Omega\mbox{,}\\
u(0) &= u_0\mbox{,} & \mbox{in~} \Omega\mbox{,}\\
G(u(T)) &\leq 0\mbox{,}\\
q_a &\leq q(t) \leq q_b\mbox{,} &\mbox{in~} \omega\mbox{,~} t\in (0,T)\mbox{.}
\end{aligned}
\right.
\end{equation}	
Here, we consider either distributed controls
\(q(t) \in L^2(\omega)\) for some appropriate subset \(\omega \subset \Omega\) of
the domain $\Omega \subset \R^d$, $d \in \{\,2,3\,\}$, or
parameter controls \(q(t) \in \R^{N_c}\), \(N_c \in \N\).
Moreover, $B$ is an appropriate control operator and
$q_a, q_b \in \R$ are the control bounds; see
\cref{sec:notation_assumptions} for the precise assumptions. The
terminal constraint on the state is given by
\begin{equation}\label{eq:terminal_constraint}
G(u) \ldef \frac{1}{2}\norm{u-u_d}_{L^2(\Omega)}^2 - \frac{\delta_0^2}{2},
\end{equation}
where $u_d$ denotes the desired state and $\delta_0 > 0$ is a given tolerance.
Thus, the goal is to steer the heat-equation from an initial heat 
distribution $u_0$ as fast as possible into a ball of radius
\(\delta_0\) around the desired state $u_d$.
Without doubt, time-optimal control is a classical subject in control theory
and we refer to, e.g., the monographs \cite{Hermes1969,Macki1982,Fattorini2005}
for a general overview.

The aim of this article is to describe, for the first time, an appropriate fully space-time discrete
version of~\eqref{P} and to prove 
\emph{a~priori} discretization error estimates. 
We note that the problem is posed on a variable time-horizon, which introduces a
nonlinear dependency on the additional variable \(T\).
Furthermore, the optimal solutions to~\eqref{P} are typically bang-bang
(i.e.\ the set where the control does not equal the control bounds is a set of zero measure),
since there are no control costs in the objective.	
This significantly
complicates the numerical analysis of~\eqref{P} compared to linear-quadratic 
problems with a fixed final time \(T\) 
considered in, e.g, \cite{Meidner2008a,Meidner2008,Meidner2011a}
with control costs in the objective
or \cite{vonDaniels2017b}
without control costs.

Even though time-optimal control problems have been extensively studied, 
there are a few publications concerning the discretization of
this problem class in the context of parabolic equations.
In \cite{Schittkowski1979,Knowles1982,Lasiecka1984,Wang2012,Zheng2014,Gong2016,Tucsnak2016,Huang2018}
the state equation is discretized in space only; 
see also the introduction of \cite{Bonifacius2017a} for a detailed comparison.
To the best of our knowledge, the only paper considering a full space-time
discretization is \cite{Bonifacius2017a} by the authors.
However, in contrast to the aforementioned articles, an
additional cost term
\begin{equation}\label{eq:intro_control_costs}
\frac{\alpha}{2}\norm{q}^2_{L^2((0,T)\times\omega)}
\quad\text{with } \alpha > 0
\end{equation}
is added to the objective functional in \cite{Bonifacius2017a}. Unfortunately, the analysis given
there does not apply in the case \(\alpha = 0\). Moreover,
the derived error estimates
depend essentially on a \emph{second order sufficient condition}, and the
constants in the error estimates explode for \(\alpha \to 0\).
For this reason we cannot directly rely on those results.	

To deal with the variable time horizon,
the state and control variables are 
transformed to a reference interval. 
The state equation is discretized by means of the
discontinuous Galerkin scheme in time and linear finite elements in space. 
We prove various convergence results;
see also \cref{table:convergence_results} for an overview.
First, we show existence of solutions to the discrete problem and
convergence of the optimal times $T_{kh} \to T$,
where we only suppose that a linearized Slater condition
on the continuous level
holds. We emphasize that the latter condition
is automatically satisfied in the setting with $u_d = 0$ and $0 \in \Qad$
as considered in \cite{Knowles1982,Wang2012,Tucsnak2016}; 
see \cite[Theorem~3.10]{Bonifacius2017a}.
Second, if the optimal control is unique, 
then we show a convergence rate for the terminal time.
For example, in the important special case of purely time-dependent controls,
we obtain the optimal convergence rate $\mathcal{O}(k + h^2)$ for the optimal time
(up to logarithmic factors).
Here, $k$ and $h$ denote the temporal and spatial mesh size, respectively.
In addition, we then assume that the following nodal set condition
\begin{equation}\label{eq:nodal-set-condition-intro}
\abs{\set{(t,x) \in I\times\omega \colon (\ControlOp^*\bar{z})(t,x) = 0 }} = 0
\end{equation}
holds. It requires the nodal set of the observation associated to the optimal adjoint state
$\bar{z}$ (see \cref{lemma:first_order_optcond}) to be of measure zero,
where $\abs{\cdot}$ denotes the measure associated with the product set $I\times\omega$.
Based on~\eqref{eq:nodal-set-condition-intro} we prove further convergence results for the controls.
Note that this condition,
which is guaranteed for, e.g., the linear heat-equation 
with a distributed control (cf.\ also \cite{Tucsnak2016}),
ensures uniqueness and the bang-bang property for the optimal control.
Here, we generalize a technique based on
a structural assumption of the adjoint state.	
Precisely, we show that the nodal set condition~\eqref{eq:nodal-set-condition-intro} implies
the existence of a continuous function $\Psi\colon [0,\infty) \to [0,\infty)$ with \(\Psi(0) = 0\) such that
\begin{equation}\label{eq:Psi_intro}
\abs{\set{(t,x) \in I\times\omega \colon -\varepsilon \leq (\ControlOp^*\bar{z})(t,x) \leq \varepsilon }} \leq \Psi(\varepsilon)
\end{equation}
holds for all $\varepsilon > 0$.
Based on~\eqref{eq:Psi_intro}, we derive an (abstract) growth condition.
Furthermore, we prove that the nodal set condition~\eqref{eq:nodal-set-condition-intro} is
a sufficient optimality condition (see \cref{thm:adjoint_based_ssc_bang_bang}),
which seems to be a new result.
Finally, assuming that the structural assumption~\eqref{eq:Psi_intro}
is valid with $\Psi(\varepsilon) = C\varepsilon^\kappa$ 
for some constants $C, \kappa > 0$,
we obtain the convergence rate $\left(k+h^2\right)^\kappa$ in $L^1$
for the control variable.
In this way, we are able to prove results which directly apply to the
global solutions of the discrete problem, without requiring that they are chosen
close to the (unique) continuous optimal solution.

Our results are improvements over existing contributions
in different aspects. First, and most importantly, 
we deal with fully discrete problem formulations, which is crucial
since it directly reflects how the problems are solved in practice.
Neglecting this fact we compare our results to the literature in the following. 
In \cite{Knowles1982} an error estimate for the optimal times is proved
that does not require uniqueness of the solution.
However, in the particular case considered there
the linearized Slater condition holds uniformly for the discrete problem,
and this would also suffice for our argument;
we also refer to \cite[Section~5.6]{Bonifacius2018} for a generalization
of \cite{Knowles1982} to fully discrete problems.
For the case of a distributed control
with the variational control discretization 
we can improve the result of \cite{Gong2016} 
(see also \cite{Zheng2014} for a semilinear state equation)
and obtain an optimal rate $\mathcal{O}(k+h^2)$.	
While the corresponding result from \cite{Wang2012} 
with an explicit control discretization
requires certain conditions $(H1)$ and $(H2)$,
which so far could only	be verified in very special situations,
we assume a condition on the set of switchings
which can be justified from practical observations; 
see \cref{thm:robust_estimate_bb_piecewise_cellwise_constant}.
In \cite{Huang2018} an error estimate of order $h^{2-\varepsilon}$
is obtained for a globally acting control
and a semilinear state equation,
whereas we can only prove $\mathcal{O}(k+h^{3/2})$. 
The reduced rate is due to control bounds in $L^\infty((0,T); L^2)$
instead of pointwise control constraints in time and space.
Last, to the best of our knowledge, this article is the first one
dealing with quantitative error estimates for the control variable in the context
of time-optimal problems, using the structural assumption~\eqref{eq:Psi_intro}.

Finally, we comment on the validity of~\eqref{eq:Psi_intro} with $\Psi(\varepsilon) =
C\varepsilon^\kappa$ for some \(\kappa \in (0,1]\).
Although it is difficult to quantify the structural assumption a priori,
we try to check it numerically, which serves as an indicator 
for the assumption for the continuous problem.
In case of purely time-dependent controls,
$\kappa = 1$ is valid in our examples and we observe the optimal order of
convergence $\mathcal{O}(k+h^2)$ for the controls in \(L^1\).
This is related to the fact that the \(N_c\) time-dependent functions
constituting \(B^*\bar{z}\) have only a finite number of simple roots;
cf.\ Remark~\ref{remark:assumption_Psi}.
In contrast, in case of a distributed control,
the structural assumption only appears to be satisfied with
$\Psi(\varepsilon) = C\varepsilon^\kappa$ for some \(\kappa < 1\)
in our numerical tests, which restricts the rate of convergence.
Here, we observe a better rate of convergence
than expected for the value of $\kappa$ that we estimated numerically.
However, the optimal theoretical value of \(\kappa\) remains an open problem.

\begin{table}
	\begin{small}
		\begin{tabularx}{\textwidth}{p{3.9cm}llX}
			\toprule
			Assumptions & $\abs{T_{kh}-T}$ & Control variable & Results\\
			\midrule
			Linearized Slater condition & $\to 0$ & -- 
			& \cref{prop:convergence_Pkh}\\
			+ uniqueness of $\bar{q}$ & $\mathcal{O}\left(k + h^2\right)$ & $\bar{q}_{kh} \rightharpoonup \bar{q}$
			& \cref{lemma:robust_error_estimate_times} \\
			+ nodal set condition~\eqref{eq:nodal-set-condition-intro} & $\mathcal{O}\left(k + h^2\right)$ & $\bar{q}_{kh} \to \bar{q}$
			& \cref{lemma:robust_error_estimate_bb_suboptimal}\\
			+ \eqref{eq:Psi_intro} with $\Psi(\varepsilon) = C\varepsilon^\kappa$ & $\mathcal{O}\left(k + h^2\right)$ & $\norm{\bar{q}_{kh}-\bar{q}}_{L^1} \lesssim \left(k+h^2\right)^\kappa$ &
			\cref{thm:robust_estimate_bb_purely_timedep_discrete},
			\cref{thm:robust_estimate_bb_variational}\\
			\bottomrule
		\end{tabularx}
	\end{small}
	\caption{Summary of convergence results neglecting logarithmic terms. 
		For simplicity we assume purely time-dependent control or distributed control with variational control discretization;
		see \cref{thm:robust_estimate_bb_piecewise_cellwise_constant} for distributed control with piecewise and cellwise constant control.}
	\label{table:convergence_results}
\end{table}

Concerning the numerical realization, we use the bilevel algorithm
from~\cite{Bonifacius2018a} that
is based on an equivalent reformulation	of \eqref{P}. In the outer loop
we employ a Newton method to find the root of a certain value function.
For the inner loop, we use an accelerated conditional gradient method.
It is worth mentioning that this approach does not require a regularization
term such as~\eqref{eq:intro_control_costs} in the objective.

This paper is organized as follows. In \cref{sec:notation_assumptions} we introduce 
the notation and main assumptions. Necessary 
and sufficient optimality conditions
are discussed in \cref{sec:time_optimal_control}.
\Cref{sec:robust_error_estimates_bb} is devoted to the discretization
of the optimal control problem and the corresponding error estimates.
In \cref{sec:examples_bb} we conclude with some numerical examples.
The proposed algorithm is sketched in \cref{sec:algorithm}.

\section{Notation and main assumptions}
\label{sec:notation_assumptions}

We generally work with the same notation and assumptions as in \cite{Bonifacius2017a}
that will be summarized in the following for the convenience of the reader.
For a Lipschitz domain $\Omega \subset \R^d$, 
let $H^1_0(\Omega)$ denote the usual Sobolev space
with zero trace on the boundary. Its dual space is $H^{-1}(\Omega)$. 
We use $\pair{\cdot,\cdot}$ to denote the duality 
pairing between $H^1_0(\Omega)$ and $H^{-1}(\Omega)$.
Usually we drop the spatial domain $\Omega$ from the notation of 
the function spaces,
if ambiguity is not to be expected.
For a Hilbert space $Z$, $\inner{\cdot,\cdot}_Z$ stands for its inner product.
Last, $c$ is a generic constant that may have different values at different appearances.

Throughout this paper we impose the following assumptions.
\begin{assumption}\label{assumption:Omega}
	Let $\Omega \subset \R^d$, $d \in \{\,2,3\,\}$, be a polygonal or polyhedral and
	convex domain.
	Moreover, the initial value satisfies $u_0 \in H^1_0(\Omega)$.
\end{assumption}
Concerning the control operator $\ControlOp$ we consider one of the following
situations:
\begin{enumerate}[(i)]
	\item Distributed control: Let $\omega \subseteq \Omega$ be the control domain that is polygonal or polyhedral as well. The control operator $\ControlOp \colon L^2(\omega) \to L^2(\Omega)$ is the extension by zero and its adjoint $\ControlOp^* \colon L^2(\Omega) \to L^2(\omega)$ is the restriction to $\omega$ operator. 
	\item Purely time-dependent control: For $N_c \in \N$, let $\omega = \set{1,2,\ldots,N_c}$
	be equipped with the counting measure. The control operator is defined by $Bq =
	\sum_{n=1}^{N_c}q_n e_n$, where $e_n \in L^2(\Omega)$ are given form functions. Then
	we have $L^2(\omega) \cong \R^{N_c}$ and $B^* \colon L^2(\Omega) \to \R^{N_c}$ with $(B^*\varphi)_n =
	\inner{e_n,\varphi}_{L^2(\Omega)}$ for $n = 1,2\ldots,N_c$.
\end{enumerate}
The space of admissible controls is defined as
\begin{equation*}
\Qad \ldef \left\{q \in \LcontrolSpatial \constraintSet  q_a \leq q \leq q_b \;\mbox{~a.e.\ in~}\; \omega\right\} \subset \LcontrolSpatial[\infty]
\end{equation*}
for $q_a, q_b \in \R$ with $q_a < q_b$.
Moreover, for $T > 0$ we set $\Q(0,T) \ldef L^2((0,T)\times\omega)$ and
\begin{equation*}
\Qad(0,T) \ldef \left\{q \in \Q(0,T) \constraintSet  q(t) \in \Qad\mbox{~a.e.~} t\in (0,T)\right\} \subset \Lcontrol[(0,T)]{\infty}.
\end{equation*}
The set $(0,T)\times\omega$ is always equipped with the completion of the product measure.
Furthermore, we use $W(0,T)$ to abbreviate
$\MPRHilbert{(0,T)}{H^{-1}}{H^1_0}$, endowed with the canonical norm and inner product. The
symbol $i_T \colon W(0,T) \rightarrow H$ denotes the continuous trace mapping $i_T u = u(T)$.
Last, the control operator $\ControlOp$ is extended to 
$\Q(0,T)$ by
$(\ControlOp q)(t) = \ControlOp q(t)$ for any $q \in \Q(0,T)$.	
\begin{assumption}\label{assumption:terminal_constraint}
	The terminal constraint $G$ is defined by~\eqref{eq:terminal_constraint}
	for a fixed desired state $u_d \in H_0^1(\Omega)$ and $\delta_0> 0$.
\end{assumption}
\begin{remark}\label{remark:regularity_G}
	The error analysis is also valid for more general terminal constraints.
	Concretely, we require that $G$ is strictly convex, two times continuously
	Fr{\'e}chet-differentiable, 
	$G''$ is bounded on bounded sets in
	$L^2$, and $G'(u)^* \in H^1_0$ for any $u\in H^1_0$.
	We focus on~\eqref{eq:terminal_constraint} to make
	the main ideas clearly visible to the reader.
\end{remark}		
In order to ensure existence of optimal solutions, 
we require
\begin{assumption}\label{assumption:existence_feasible_control}
	There exist a finite time $T > 0$ and a feasible control $q \in \Qad(0,T)$ such that the solution to the state equation of \eqref{P} satisfies $G(u(T)) \leq 0$. To exclude the trivial case, we additionally assume $G(u_0) > 0$.
\end{assumption}
\begin{proposition}
	\label{prop:existence_optimal_control}
	If \cref{assumption:existence_feasible_control} holds, 
	then there exists a globally optimal solution 
	$(T,\bar{q}) \in \Rplus\times\Qad(0,T)$	to \eqref{P}.
	Moreover, the final time $T$ and 
	the observation $\bar{u}(T)$ are unique.
\end{proposition}
\begin{proof}
	Existence follows by the direct method;
	cf.\ \cite[Proposition~3.1]{Bonifacius2017a}.
	Moreover, $T$ is unique, because $T$ is the objective functional.
	Last, uniqueness of $\bar{u}(T)$ follows
	from strict convexity of $G$
	and linearity of the control-to-observation mapping.
\end{proof}		
We also refer to \cite[Remark~2.2]{Bonifacius2017a} for a discussion of several
situations where \cref{assumption:existence_feasible_control} is guaranteed to hold.

\section{The time-optimal control problem}
\label{sec:time_optimal_control}
In this section we introduce the transformation approach, which forms the basis of
the discretization concept, and collect results for the continuous problem~\eqref{P} which are
fundamental for the error analysis.

\subsection{Change of variables}
We transform the state equation to a
fixed reference time interval in order to deal with the variable time horizon of~\eqref{P}.
For $\nu \in \Rplus$ we set $T_\nu(t) = \nu t$ and obtain the transformed state equation
\begin{equation*}
\partial_t u - \nu\Lap u = \nu \ControlOp q\mbox{,}\quad
u(0) = u_0\mbox{.}
\end{equation*}
For each pair $(\nu,q) \in \Rplus\times\Q(0,1)$ there exists a unique solution to the transformed state equation;
see, e.g.,~\cite[Theorem~2, Chapter XVIII, \S3]{Dautray1992}.
Abbreviating $I = (0,1)$, let $S \colon \Rplus\times\Qad(0,1) \rightarrow W(I)$, $(\nu,q)\mapsto u$, denote the corresponding control-to-state mapping.
We define the reduced terminal constraint by
\begin{equation*}
g(\nu,q) \ldef G(i_1S(\nu,q))\mbox{,}
\end{equation*}
where $i_1$ denotes the trace mapping.
The transformed optimal control problem is
\begin{equation}\label{Pt}\tag{\mbox{$\hat{P}$}}
\mbox{Minimize~} \nu \mbox{~subject to~}g(\nu,q) \leq 0, (\nu,q) \in \Rplus\times\Qad(0,1).
\end{equation}
Note that both problems~\eqref{P} and~\eqref{Pt} 
are equivalent; see, e.g., \cite[Proposition~4.6]{Bonifacius2017}.
Moreover, continuity of the trajectory $u \colon [0,1] \to L^2$ implies 
that the inequality constraint in~\eqref{Pt} can be replaced 
by an equality constraint, i.e.\ $g(\nu,q) = 0$.
Otherwise a feasible control with a shorter time exists, which contradicts the optimality of the solution.
Throughout the paper, we will need the following differentiability property,
which is obtained by standard arguments; cf.\ also \cite[Section~3.1]{Bonifacius2017a}.
\begin{lemma}\label{lemma:control_to_state_differentiable}
	Let $\nu \in \Rplus$ and $q \in \Q(0,1)$.
	The control-to-state mapping $S$ is twice continuously Fr\'echet-differentiable. Moreover, $\delta u = S'(\nu,q)(\delta \nu, \delta q) \in W(0,1)$ is the unique solution to
	\begin{equation*}\label{eq:linearized_state_equation}
	\partial_t \delta u - \nu\Lap\delta u = \delta\nu (\ControlOp q+\Lap u)+ \nu \ControlOp\delta q\mbox{,}\quad
	\delta u(0) = 0\mbox{,}
	\end{equation*}
	for $(\delta \nu, \delta q) \in \R\times\Lcontrol{2}$ and $\delta\tilde{u} = S''(\nu,q)(\delta\nu_1, \delta q_1;\delta \nu_2, \delta q_2) \in W(0,1)$ is the unique solution to
	\begin{equation*}\label{eq:linearized_state_equation2}
	\partial_t \delta\tilde{u} - \nu\Lap\delta\tilde{u} = \delta\nu_1\left(\ControlOp\delta q_2 + \Lap\delta u_2\right)+ \delta\nu_2 \left(\ControlOp\delta q_1 +\Lap\delta u_1\right)\mbox{,}\quad
	\delta\tilde{u}(0) = 0\mbox{,}
	\end{equation*}
	for $(\delta \nu_i, \delta q_i) \in \R\times\Lcontrol{2}$ and $\delta u_i = S'(\nu,q)(\delta \nu_i, \delta q_i)$, $i = 1,2$.
\end{lemma}	
By means of \cref{lemma:control_to_state_differentiable}, the reduced constraint mapping
$g\colon \Rplus\times\Q(0,1) \rightarrow \R$ is twice continuously
Fr{\'e}chet-differentiable. Moreover, the expressions
\begin{align}
g'(\nu,q)(\delta{\nu}, \delta{q})
&= \inner{u(1) - u_d,\delta{u}(1)}_{L^2},\label{eq:reduced_constraint_explicit_d1} \\
g''(\nu,q)(\delta{\nu}_1, \delta{q}_1; \delta{\nu}_2, \delta{q}_2)
&= \inner{\delta{u}_1(1),\delta{u}_2(1)}_{L^2} + \inner{u(1) - u_d,\delta{\tilde{u}}(1)}_{L^2},\label{eq:reduced_constraint_explicit_d2}
\end{align}
hold,
where \(\delta{u}_1\), \(\delta{u}_2\), and \(\delta{\tilde{u}}\) are defined as in
\cref{lemma:control_to_state_differentiable}.
Last, for $\nu \in \Rplus$, $q \in \Q(0,1)$, \(u = S(\nu,q)\), and $\mu \in \R$ 
we have the representation
\begin{equation}\label{eq:terminal_constraint_adjoint}
\mu \, g'(\nu,q)^* = \left(
\begin{array}{l}
\int_{0}^{1}\pair{Bq + \Lap u, z}\\
\nu B^*z
\end{array} \right)\mbox{,}
\end{equation}
where $z \in W(0,1)$ is the unique solution to the adjoint state equation
\begin{equation*}
-\partial_t z - \nu\Lap z = 0\mbox{,}\quad 
z(1) = \mu(u(1) - u_d)\mbox{.}
\end{equation*}

\subsection{First order necessary optimality conditions}
We summarize first order optimality conditions 
from \cite[Section~3.2]{Bonifacius2017a} that also hold in
the case without control costs in the objective functional.
Let $(\bar{\nu},\bar{q}) \in \Rplus\times\Qad(0,1)$ be a solution to~\eqref{Pt}.
We suppose that the following \emph{linearized Slater} condition is satisfied.
\begin{assumption}
	\label{assumption:linearized_slater}
	We assume that
	\begin{equation}\label{eq:definition_linearized_slater_condition}
	\bar{\eta} \ldef - \partial_\nu g(\bar{\nu}, \bar{q}) > 0\mbox{.}
	\end{equation}
\end{assumption}
Note that due to \cref{assumption:linearized_slater} and \(g(\bar{\nu}, \bar{q}) = 0\), the
point \(\breve{\chi}^\gamma = (\bar{\nu} + \gamma, \bar{q}) \in \Rplus\times\Qad(0,1)\)
defined for \(\gamma > 0\) satisfies
\[
g(\bar{\chi}) + g'(\bar{\chi})(\breve{\chi}^\gamma - \bar{\chi}) = -\bar{\eta}\,\gamma < 0\mbox{,}
\]
which is the typically used linearized Slater condition. 
Thus, we assume that this condition holds in a special form.
As discussed in \cite[Section~3.2]{Bonifacius2017a} 
\cref{assumption:linearized_slater} is already 
equivalent to qualified first order optimality conditions.
Hence, it is not restrictive to assume that 
the linearized Slater condition holds in the form~\eqref{eq:definition_linearized_slater_condition}.

To state optimality conditions, we introduce the Lagrange function as
\begin{equation*}
\mathcal{L} \colon \Rplus\times\Q(0,1)\times\R \rightarrow \R\mbox{,}\quad
\mathcal{L}(\nu, q, \mu) \ldef \nu + \mu\, g(\nu, q)\mbox{.}
\end{equation*}
Now, optimality conditions for~\eqref{Pt} in qualified form can be given as follows: 
For \(\bar{\nu} > 0\) and \(\bar{q} \in \Qad(0,1)\) being a solution to~\eqref{Pt}
there exists a \(\bar{\mu}\geq 0\), such that
\begin{equation}\label{eq:optimalityCondLagrange}
\partial_{(\nu,q)}\mathcal{L}(\bar{\nu},\bar{q},\bar{\mu})(\delta{\nu},q-\bar{q}) \geq 0
\quad \text{for all }(\delta{\nu},q) \in \R\times\Qad(0,1)\mbox{.}
\end{equation}
\Cref{assumption:linearized_slater} ensures the existence of a multiplier $\bar{\mu}$
that is always positive due to the special structure of the problem.
\begin{lemma}
	\label{lemma:first_order_optcond}
	Let $(\bar{\nu},\bar{q}) \in \Rplus\times\Qad(0,1)$ be a solution of~\eqref{Pt}
	with associated state $\bar{u} = S(\bar{\nu},\bar{q})$ and the linearized Slater
	condition~\eqref{eq:definition_linearized_slater_condition} hold. Then there exists
	a multiplier $0 < \bar{\mu} \leq c/\bar{\eta}$ such that
	\begin{align}
	\int_{0}^{1} 1 + \pair{\ControlOp\bar{q}(t) + \Lap\bar{u}(t), \bar{z}(t)}\D{t} &= 0\mbox{,}\label{eq:opt_cond_hamiltonianConstant}\\
	\int_0^{1}\pair{\ControlOp^*\bar{z}(t), q(t) - \bar{q}(t)}\D{t} &\geq 0 \quad \text{for all~} q \in\Qad(0,1)\mbox{,}\label{eq:opt_cond_variationalInequality}\\
	G(\bar{u}(1)) &= 0\mbox{,}\label{eq:opt_cond_feasiblity}
	\end{align}
	where the \emph{adjoint state} $\bar{z} \in W(0,1)$ is determined by
	\begin{equation}\label{eq:adjoint_state_equation}
	-\partial_t \bar{z}(t) - \bar{\nu}\Lap\bar{z}(t) = 0\mbox{,}
	\quad t \in (0,1) \quad 
	\bar{z}(1) = \bar{\mu}(\bar{u}(1) - u_d)\mbox{.}
	\end{equation}
\end{lemma}
\begin{proof}
	Note first that the linearized Slater condition allows for 
	exact penalization of~\eqref{Pt}; see 
	\cite[Theorem~2.87, Proposition~3.111]{Bonnans2000}.
	The optimality conditions now follow as in the proof of \cite[Theorem~4.12]{Bonifacius2017}.
	The condition~\eqref{eq:opt_cond_hamiltonianConstant} is equivalent to
	\(\partial_\nu \mathcal{L}(\bar{\nu},\bar{q},\bar{\mu}) = 0\)
	and~\eqref{eq:opt_cond_variationalInequality} arises
	from~\eqref{eq:optimalityCondLagrange} for \(\delta{\nu} = 0\).
	Last, we observe that \(\bar{\mu} = 0\) implies \(\bar{z} = 0\), which
	contradicts~\eqref{eq:opt_cond_hamiltonianConstant}.
	Thus, \(\bar{\mu} > 0\) must hold.
\end{proof}
We emphasize that the adjoint state from \cref{lemma:first_order_optcond}
is unique 
up to multiplication by the positive scalar $\bar{\mu}$
due to uniqueness of the observation; see \cref{prop:existence_optimal_control}.

From the variational inequality~\eqref{eq:opt_cond_variationalInequality} we infer that
\begin{equation}\label{eq:signcondition_adjoint_control_bb}
(\ControlOp^*\bar{z})(t,x) \begin{cases}
\geq 0 &\text{for } \bar{q}(t,x) = q_a,\\
\leq 0 &\text{for } \bar{q}(t,x) = q_b,\\
= 0 &\text{for } q_a < \bar{q}(t,x) < q_b.\\
\end{cases}
\end{equation}	
In this article we are interested in the case when $\bar{q}$ is a bang-bang
control, which is implied by the following condition:
\begin{assumption}
	\label{assumption:control_bang_bang_measure_condition}
	We assume that the \emph{nodal set condition}
	\begin{equation*}
	\abs{\set{(t,x) \in I\times\omega \colon (\ControlOp^*\bar{z})(t,x) = 0 }} = 0
	\end{equation*}
	holds,
	where $\abs{\cdot}$ denotes the measure associated with $I\times\omega$.
\end{assumption}    
\begin{proposition}
	\label{prop:control_bangbang_unique}
	If \cref{assumption:control_bang_bang_measure_condition} holds,
	then $\bar{q}$ is bang-bang and unique.
\end{proposition}
\begin{proof}
	From the optimality condition~\eqref{eq:signcondition_adjoint_control_bb}
	and \cref{assumption:control_bang_bang_measure_condition}
	we immediately infer that $\bar{q}$ is a bang-bang control.
	To show uniqueness, let $q \in \Qad(0,1)$ be a different optimal control.
	Set $q_\lambda = \lambda q + (1-\lambda)\bar{q} \in \Qad(0,1)$ for any $\lambda \in [0,1]$.
	Affine linearity of the control-to-state mapping for fixed $\nu$
	and convexity of the terminal constraint imply that 
	the pair $(\bar{\nu}, q_\lambda)$ is also feasible for \eqref{Pt}.
	In addition, a simple contradiction argument reveals that
	$(\bar{\nu}, q_\lambda)$ is also optimal, i.e.\ $g(\bar{\nu}, q_\lambda) = 0$.
	Hence,
	\[
	\bar{\nu}\int_{0}^{1}\int_\omega \ControlOp^*\bar{z} (q-\bar{q})
	= \partial_q\mathcal{L}(\bar{\nu},\bar{q},\bar{\mu})(q-\bar{q})
	= \lim_{\substack{\lambda \in (0,1]\\ \lambda \to 0}} \frac{1}{\lambda}\left[\mathcal{L}(\bar{\nu}, q_\lambda, \bar{\mu}) - \mathcal{L}(\bar{\nu}, \bar{q}, \bar{\mu})\right] = 0.
	\]
	With \eqref{eq:signcondition_adjoint_control_bb}
	and $q \in \Qad(0,1)$,
	the integrand on the left-hand side is nonnegative.
	Thus, using $\bar{\nu} > 0$,
	it is zero almost everywhere.
	Finally, \cref{assumption:control_bang_bang_measure_condition}
	implies $q=\bar{q}$, so $\bar{q}$ is unique.
\end{proof}
\begin{remark}
	We comment on situations in which \cref{assumption:control_bang_bang_measure_condition}
	is guaranteed to hold.
	\begin{enumerate}[(i)]
		\item In the case of a distributed control on an open subset $\omega \subset \Omega$,
		\cref{assumption:control_bang_bang_measure_condition} is satisfied;
		see \cite[Theorem~4.7.12]{Fattorini2005}.
		Note that \cite[Theorem~1.1]{Han1994} is only applied 
		for interior subsets
		of the cylinder $I\times\Omega$. Employing interior regularity
		of the solution to the heat-equation with zero right-hand side,
		the general boundary regularity of this article is sufficient for the argument.
		
		\item Suppose purely time-dependent controls,
		i.e.\ $\ControlOp \colon \R^{\parameterControlDim} \to H^{-1}$, 
		$\ControlOp q = \sum_{i=1}^{\parameterControlDim} q_i e_i$
		and set $\ControlOp_i \colon \R \to H^{-1}$, $\ControlOp_iq = qe_i$.
		If $(-\Lap, B_i)$ is approximately controllable for all $i = 1,2,\ldots,M$,
		i.e.\ $q \mapsto \int_{0}^{1}\semigroup{(1-s)\Lap}B_iq(s)\D{s}$ has dense range in $L^2(\Omega)$,
		then \cref{assumption:control_bang_bang_measure_condition} holds.
		This follows from analyticity of the semigroup generated by $\Lap$
		and \cite[Theorem~11.2.1, Definition~6.1.1]{Tucsnak2009}.
		In the context of time-optimal control of ODEs,
		approximate controllability of $(-\Lap, B_i)$ for all $i$ 
		is referred to as normality;
		see, e.g., \cite[Section~II.16]{Hermes1969} or \cite[Section~III.3]{Macki1982}.
		
		We note that the assumption of normality implies that the Dirichlet
		Laplacian on the domain \(\Omega\) has simple spectrum (all eigenvalues
		have geometric multiplicity one); see, e.g.,
		\cite[Theorem~1.3]{Badra2014}.
		Unfortunately, this is not fulfilled for all domains
		-- we refer to \cite[Section~3.4]{Badra2014} for a thorough discussion.
		While this limits the applicability of the above criterion to certain
		domains, we emphasize that it is only a sufficient condition.
		
	\end{enumerate}
\end{remark}

\subsection{Sufficient optimality conditions for bang-bang controls}
Let $(\bar{\nu}, \bar{q}) \in \Rplus\times\Qad(0,1)$ such that the
necessary optimality conditions from \cref{lemma:first_order_optcond} hold
with $\bar{z} \in W(0,1)$ the adjoint state $\bar{\mu} > 0$ the Lagrange multiplier.

\begin{proposition}
	\label{prop:bang-bang_yields_structural_condition}
	Let~\cref{assumption:control_bang_bang_measure_condition} hold.
	Then there exists a concave, continuous, strictly monotonically increasing 
	function $\Psi \colon [0,\infty) \to [0,\infty)$
	with $\Psi(0) = 0$ and $\lim_{\varepsilon \to \infty} \Psi(\varepsilon) =
	\infty$ such that for all $\varepsilon > 0$ it holds
	\begin{equation}\label{eq:assumption_structure_adjoint_ssc_bb}
	\abs{\set{(t,x) \in I\times\omega \colon -\varepsilon \leq (\ControlOp^*\bar{z})(t,x) \leq \varepsilon }} \leq \Psi(\varepsilon).
	\end{equation}
\end{proposition}
\begin{proof}
	Define \(\Phi(\varepsilon)
	\ldef \abs{\set{(t,x) \in I\times\omega \colon -\varepsilon \leq (\ControlOp^*\bar{z})(t,x) \leq \varepsilon }}\),
	which is the left-hand side of~\eqref{eq:assumption_structure_adjoint_ssc_bb}.
	Then, $0 \leq \Phi(\varepsilon) \leq \abs{I\times\omega} < \infty$.
	Moreover, one easily shows that \(\Phi\) is continuous from the right,
	thus, in particular we have \(\lim_{\varepsilon\searrow 0} \Phi(0) = 0\).
	Now, we define the concave hull of $\Phi$ by
	$\widetilde\Phi = -\left((-\Phi)^*\right)^*$,
	where \((\cdot)^*\) denotes the convex conjugate (also known as Fenchel-Legendre transform).
	Concretely,
	\begin{align*}
	\widetilde\Phi(\varepsilon) &= -\sup_{\gamma \in \R} \left[\varepsilon\gamma
	- \sup_{\varepsilon' \geq 0} (\varepsilon'\gamma + \Phi(\varepsilon')) \right]
	= \inf_{\gamma \geq 0} \left[\varepsilon\gamma
	+ \sup_{\varepsilon' \geq 0} (\Phi(\varepsilon') - \varepsilon'\gamma) \right],
	\end{align*}
	where, in the last line we have substituted \(\gamma\) by \(-\gamma\).
	By the standard properties of the concave hull (see, e.g., \cite[Corollary~4.22]{Clarke2013}), we have
	\begin{equation}\label{eq:bang-bang_yields_structural_condition_P2}
	\widetilde\Phi(\varepsilon) \geq \Phi(\varepsilon)\quad \text{for all~} \varepsilon \geq 0,
	\end{equation}
	and \(\widetilde\Phi\) is upper semi-continuous.
	Furthermore, we can verify that
	\[
	\widetilde\Phi(0) = \inf_{\gamma \geq 0} \sup_{\varepsilon' \geq 0} (\Phi(\varepsilon') - \varepsilon'\gamma)
	= 0.
	\]
	First, \(\widetilde\Phi(0) \geq 0\) follows from \eqref{eq:bang-bang_yields_structural_condition_P2}
	and \cref{assumption:control_bang_bang_measure_condition}.
	Assume that \(\widetilde\Phi(0) > 0\).
	Then for each $\gamma > 0$
	there is $\varepsilon = \varepsilon(\gamma) > 0$
	such that
	$\widetilde\Phi(0) \leq \Phi(\varepsilon) - \varepsilon\gamma + \widetilde\Phi(0)/2$.
	Hence, $\widetilde\Phi(0)/2 \leq \Phi(\varepsilon)$ and
	$\varepsilon < \Phi(\varepsilon)/\gamma$. 
	Using the boundedness of $\Phi$,
	we find $\varepsilon(\gamma) \to 0$ for $\gamma \to \infty$.
	However, $\lim_{\varepsilon \searrow 0} \Phi(\varepsilon) = 0$
	which is a contradiction to $\widetilde\Phi(0)/2 \leq \Phi(\varepsilon)$.
	Finally, \(\widetilde\Phi\) is continuous, since it is 
	Lipschitz-continuous on \((0,\infty)\) (see, e.g.,
	\cite[Theorem~2.34]{Clarke2013}) and
	\(\lim_{\varepsilon \to 0} \widetilde\Phi(\varepsilon) = 0\) with
	upper semi-continuity.	
	We conclude the proof by setting \(\Psi(\varepsilon) = \widetilde\Phi(\varepsilon) + \varepsilon\)
	to guarantee strict monotonicity and $\lim_{\varepsilon \to \infty} \Psi(\varepsilon) = \infty$.
\end{proof}

\begin{remark}
	\label{remark:assumption_Psi}
	\begin{enumerate}[(i)]
		\item In related contexts, the condition~\eqref{eq:assumption_structure_adjoint_ssc_bb} 
		is an assumption; see, e.g., \cite{Wachsmuth2011,Deckelnick2012,Wachsmuth2013,Casas2016,vonDaniels2017,vonDaniels2017b,Casas2017a}
		for the special case $\Psi(\varepsilon) = C\varepsilon^\kappa$ 
		with constants $C > 0$ and $\kappa > 0$.
		However, we derived the existence of such a function $\Psi$, requiring only the nodal set condition from \cref{assumption:control_bang_bang_measure_condition},
		which is guaranteed in many examples.
		
		\item In the context of a distributed control, where
		\(B^*\bar{z} = \bar{z}|_{I\times\omega}\), a sufficient condition for a strong form
		of~\eqref{eq:assumption_structure_adjoint_ssc_bb} is often given as follows (see, e.g.,
		\cite[Lemma~3.2]{Deckelnick2012}):
		Assume $\bar{z} \in C^1(\overline{I\times\omega})$ and that there exists a
		constant $c > 0$ such that
		\[
		\norm{\nabla_{(t,x)}\bar{z}(t,x)}_{\R^{d+1}} \geq c
		\]
		for all $(t,x) \in I\times\omega$ such that $\bar{z}(t,x) = 0$, then
		\eqref{eq:assumption_structure_adjoint_ssc_bb} holds with
		$\Psi(\varepsilon) = C\varepsilon$.
		
		\item Condition~\eqref{eq:assumption_structure_adjoint_ssc_bb} is also
		compatible with purely time-dependent controls. In this case  the
		structural condition concretely reads as
		\[
		\sum_{n=1}^{N_c} \abs{\set{t \in I \colon \abs{(B^*\bar{z}(t))_n} \leq \varepsilon}}
		\leq \Psi(\varepsilon).
		\]
		In the context of optimal control problems with ODEs, the functions
		\(t \mapsto \sigma_n(t) = (B^*\bar{z}(t))_n = (e_n,\bar{z}(t))_{L^2(\Omega)}\) are
		referred to as switching functions. Here, one typically
		assumes that each \(\sigma_n\) has only finitely many roots with non-vanishing first derivatives
		(see, e.g., \cite{Felgenhauer2003,Maurer2004}),
		which again implies \eqref{eq:assumption_structure_adjoint_ssc_bb} with
		$\Psi(\varepsilon) = C\varepsilon$.
	\end{enumerate}
\end{remark}
Clearly, any function \(\Psi\) with the properties as given in
\cref{prop:bang-bang_yields_structural_condition} possesses a convex,
strictly monotonously increasing and continuous inverse 
\(\Psi^{-1}\colon [0,\infty) \to [0,\infty)\) with \(\Psi^{-1}(0) = 0\) and \(\lim_{x\to\infty}
\Psi^{-1}(x) = \infty\).
The proof of sufficiency of the structural assumption for a pair $(\bar{\nu},\bar{q})$ to be locally optimal 
relies now on the following result.
\begin{proposition}\label{prop:langrange_lowerbound_quadratic_l1}
	Let $(\bar{\nu}, \bar{q}) \in \Rplus\times\Qad(0,1)$ satisfy the
	necessary optimality conditions from \cref{lemma:first_order_optcond}.
	Moreover, suppose that~\cref{assumption:control_bang_bang_measure_condition} holds. Then there is $c_0 > 0$ such that
	\begin{equation}\label{eq:langrange_lowerbound_quadratic_l1}
	\partial_{q}\mathcal{L}(\bar{\nu},\bar{q},\bar{\mu})(q-\bar{q}) \geq \frac{\bar{\nu}}{2}\Psi^{-1}\left(c_0\norm{q-\bar{q}}_{L^1(I\times\omega)}\right)\norm{q-\bar{q}}_{L^1(I\times\omega)}
	\end{equation}
	for all $q \in \Qad(0,1)$,
	where $\Psi$ is from \cref{prop:bang-bang_yields_structural_condition}.
\end{proposition}
\begin{proof}
	The proof is along the lines of \cite[Proposition~2.7]{Casas2016}
	with slight modifications.
	For $q \in \Qad(0,1)$, we take
	$\varepsilon \ldef \Psi^{-1}(\left(2 \abs{q_b-q_a}\right)^{-1} \norm{q-\bar{q}}_{L^1(I\times\omega)}).$
	Now the estimate follows as in \cite{Casas2016} with $c_0 = \left(2\abs{q_b-q_a}\right)^{-1}$.
\end{proof}

\Cref{assumption:control_bang_bang_measure_condition} allows to prove the following 
growth condition without two norm discrepancy. 
In particular, we infer that the nodal set condition
is a sufficient optimality condition 
for the time-optimal control problem~\eqref{P}.	
It is worth mentioning that due to the particular objective functional we do not require additional assumptions 
such as conditions on the second derivative of the Lagrange function; cf.\ \cite[Theorem~2.2]{Casas2012a} and \cite[Theorem~2.8]{Casas2016}.
\begin{theorem}\label{thm:adjoint_based_ssc_bang_bang}
	Let $(\bar{\nu}, \bar{q}) \in \Rplus\times\Qad(0,1)$ satisfying first order necessary optimality conditions. 
	Moreover, suppose that~\cref{assumption:control_bang_bang_measure_condition} holds. 
	Then $(\bar{\nu}, \bar{q})$ is optimal for~\eqref{Pt} and
	there exists a constant $\delta > 0$ such that
	\begin{equation}\label{eq:quadratic_growth_bb}
	\frac{\bar{\nu}}{6}\Psi^{-1}\left(c_0 \norm{q-\bar{q}}_{L^1(I\times\omega)}\right)\norm{q-\bar{q}}_{L^1(I\times\omega)} 
	\leq \nu - \bar{\nu}
	\end{equation}
	for all admissible $(\nu, q) \in \Rplus\times\Qad(0,1)$, i.e.\ $g(\nu,q) \leq 0$,
	with $\abs{\nu-\bar{\nu}} \leq \delta$.
\end{theorem}
\begin{remark}
	In particular, the result from \cref{thm:adjoint_based_ssc_bang_bang}
	implies that \(\nu > \bar{\nu}\) for any admissible \((\nu,q)\) with \(q \neq
	\bar{q}\), i.e., the optimal control \(\bar{q}\) is unique.
\end{remark}
In order to prove the result, we first observe that the second derivative of the Lagrange function can be bounded below as follows.
\begin{proposition}\label{prop:adjoint_based_ssc_lowerbound_hesse}
	Let $(\bar{\nu}, \bar{q}) \in \Rplus\times\Qad(0,1)$, $\bar{\mu} > 0$,
	and $0 < \nu_{\min} < \nu_{\max}$. There is $c > 0$ such that
	\begin{equation*}
	\partial^2_{(\nu,q)}\mathcal{L}(\nu_\xi,q_\xi,\bar{\mu})[\nu-\bar{\nu},q-\bar{q}]^2 \geq -c\abs{\nu-\bar{\nu}}^2 - c\abs{\nu-\bar{\nu}}\norm{q-\bar{q}}_{L^2(I\times\omega)}
	\end{equation*}
	for all $\nu, \nu_\xi \in \Rplus$, $q, q_\xi \in \Qad(0,1)$ with
	$\nu_{\min} \leq \nu, \nu_\xi \leq \nu_{\max}$.
\end{proposition}
\begin{proof}
	Set $\delta\nu = \nu-\bar{\nu}$ and $\delta q = q-\bar{q}$.
	Define $u_\xi = S(\nu_\xi,q_\xi)$, $\delta u = S'(\nu_\xi,q_\xi)(\delta\nu,\delta q)$, and $\delta\tilde{u} = S''(\nu_\xi,q_\xi)[\delta\nu,\delta q]^2$. Moreover, let $z_\xi$ be the corresponding adjoint state with terminal value $\bar{\mu}(u_\xi(1)-u_d)$.
	Then we observe
	\begin{multline*}
	\bar{\mu}\inner{u_\xi(1) - u_d, \delta\tilde{u}(1)}_{L^2}
	= \inner{z_\xi(1),\delta\tilde{u}(1)}_{L^2} - \inner{z_\xi(0),\delta\tilde{u}(0)}_{L^2}\\
	= \int_{0}^{1}\pair{\partial_t \delta \tilde{u},z_\xi} +  \int_{0}^{1}\pair{\partial_t z_\xi, \delta \tilde{u}}
	=  \int_{0}^{1}\pair{\partial_t \delta\tilde{u},z_\xi} - \int_{0}^{1}\pair{\bar{\nu}\Lap\delta\tilde{u}, z_\xi}\\
	= 2\delta\nu\int_{0}^{1}\pair{\ControlOp\delta q+\Lap\delta u, z_\xi}\D{t}.
	\end{multline*}
	Thus, using \eqref{eq:reduced_constraint_explicit_d2}, we find
	\begin{align*}
	\partial^2_{(\nu,q)}\mathcal{L}(\nu_\xi,q_\xi,\bar{\mu})[\delta\nu,\delta q]^2
	&= \bar{\mu}\norm{\delta u(1)}^2_{L^2}
	+ 2\delta\nu\int_{0}^{1}\pair{\ControlOp\delta q+\Lap\delta u, z_\xi}\D{t}\nonumber\\
	&\geq -2\abs{\delta\nu}\int_{0}^{1}\abs{\pair{\ControlOp\delta q+\Lap\delta u, z_\xi}}\D{t}.
	\end{align*}
	The Cauchy-Schwarz inequality and
	the stability estimates for $u_\xi, \delta u$, and $z_\xi$
	with \cref{lemma:control_to_state_differentiable} further imply
	\begin{multline*}
	\partial^2_{(\nu,q)}\mathcal{L}(\nu_\xi,q_\xi,\bar{\mu})[\delta\nu,\delta q]^2 
	\geq -2\abs{\delta\nu}
	\left(\norm{\ControlOp\delta q}_{L^2(I;H^{-1})} + \norm{\delta u}_{L^2(I;H^1_0)}\right)\norm{z_\xi}_{L^2(I;H^1_0)}\\
	\geq -c\abs{\delta\nu}\left(\norm{\ControlOp\delta q}_{L^2(I;H^{-1})} + \frac{\abs{\delta\nu}}{\nu_\xi}\left(\norm{\ControlOp q_\xi}_{L^2(I;H^{-1})} + \norm{u_\xi}_{L^2(I;H^1_0)}\right) \right)
	\norm{z_\xi(1)}_{L^2}.
	\end{multline*}	
	Since $q_\xi$ is uniformly bounded due to boundedness of $\Qad(0,1)$ 
	and $\nu_\xi$ is uniformly bounded from below and from above,
	there exists a constant $c > 0$ such that
	\[
	\partial^2_{(\nu,q)}\mathcal{L}(\nu_\xi,q_\xi,\bar{\mu})[\delta\nu,\delta q]^2 
	\geq -c\abs{\delta\nu}^2 - c\abs{\delta\nu}\norm{\delta q}_{L^2(I\times\omega)}
	\]
	proving the assertion.
\end{proof}	
Last, we require a technical result,
which follows from the Fenchel-Young inequality.
\begin{proposition}
	\label{prop:young_inequality_Psiinverse}
	Let $\varepsilon > 0$, $c_0 > 0$,
	and let $\Psi$ satisfy the assumptions of~\cref{prop:bang-bang_yields_structural_condition}.
	Then there exists a (convex) function $F \colon [0,\infty) \to [0,\infty)$
	such that
	\[
	x y \leq \varepsilon\Psi^{-1}(c_0 x^2)x^2 + F(y)
	\quad\text{for all } x, y \in [0,\infty),
	\]
	and $F(0) = 0$ and $\lim_{y \to 0} F(y)/y = 0$.
\end{proposition}
\begin{proof}
	We abbreviate $H(x) = \varepsilon\Psi^{-1}(c_0 x^2)x^2$.
	Note first that \(\Psi^{-1}\) is convex as the inverse of a
	concave function. Thus, it is Lipschitz continuous on the
	interior of its domain \(\Rplus\); see, e.g., \cite[Theorem~2.34]{Clarke2013}.
	Therefore, we can apply Rademacher's theorem and the
	chain rule to compute the derivative 
	\[
	H'(x) = 2\varepsilon\left(c_0 (\Psi^{-1})'(c_0 x^2)x^3 + \Psi^{-1}(c_0 x^2)x\right),
	\]
	which is defined almost everywhere.
	Using again that \(\Psi^{-1}\) is convex, we verify that $H'$ is monotonically increasing.
	Hence, $H$ is convex, locally Lipschitz continuous, and strictly monotonically increasing.
	Now, we define \(F = H^*\), where \(H^*(y) =
	\sup_{x\geq 0} \left[ yx - H(x) \right]\) is the
	convex conjugate of \(H\). Clearly, \(F(y) \geq F(0) = H(0) = 0\) for all \(y \geq 0\).
	Thus, the desired inequality is
	given by the Fenchel-Young inequality \(xy \leq H(x) + H^*(y)\).
	
	It remains to verify that the directional derivative of \(F\) at zero, i.e.\
	$F'(0,+1) = \lim_{y \to 0} (F(y)-F(0))/(y-0) = \lim_{y \to 0} F(y)/y$, is
	equal to zero. 
	We consider the subdifferential \(\partial F(0)\), which reads in this case as               
	\[
	\partial F(0) = \set{v \in \R \colon F(x)/x \geq v \text{~for all~} x > 0}.
	\]
	Assume that \(F'(0,+1) > 0\). 
	Then \((-\infty, F'(0,+1)] \subset \partial F(0)\).
	Therefore, we deduce that \(0 \in \partial H(x)\) for all \(x\leq
	F'(0,+1)\) by the subdifferential inversion formula; see, e.g.,
	\cite[Exercise~4.27]{Clarke2013}. This implies that these points are
	global minima of \(H\) and thus \(H(x) = 0\) for \(x\leq F'(0,+1)\), which contradicts the
	strict monotonicity of \(H\).
\end{proof}
Finally, we are give the proof of the main result of this section.
\begin{proof}[Proof of \cref{thm:adjoint_based_ssc_bang_bang}.]
	Let $(\nu, q) \in \Rplus\times\Qad(0,1)$ be admissible. Set $\delta\nu = \nu -\bar{\nu}$ and $\delta q = q - \bar{q}$.
	Using feasibility of $(\nu,q)$, the facts that $\bar{\mu} > 0$ and $g(\bar{\nu},\bar{q}) = 0$ 
	from the necessary optimality conditions for $(\bar{\nu},\bar{q})$, as well as Taylor expansion we find
	\begin{align*}
	\nu - \bar{\nu} &\geq \nu + \bar{\mu}g(\nu,q) - \left(\bar{\nu} + \bar{\mu}g(\bar{\nu},\bar{q})\right)
	= \mathcal{L}(\nu,q,\bar{\mu}) -\mathcal{L}(\bar{\nu},\bar{q},\bar{\mu})\\
	&= \partial_{(\nu,q)}\mathcal{L}(\bar{\nu},\bar{q},\bar{\mu})(\delta\nu,\delta q)
	+ \frac{1}{2}\partial^2_{(\nu,q)}\mathcal{L}(\nu_\xi,q_\xi,\bar{\mu})[\delta\nu,\delta q]^2,
	\end{align*}
	with appropriate $\nu_\xi = \bar{\nu} + \xi_\nu(\nu-\bar{\nu})$ and $q_\xi = \bar{q} + \xi_q(q-\bar{q})$ 
	for $0 \leq \xi_\nu, \xi_q \leq 1$.
	Thus, according to \cref{prop:adjoint_based_ssc_lowerbound_hesse} there is $c_1 > 0$ such that
	\[
	\nu - \bar{\nu} \geq \partial_{(\nu,q)}\mathcal{L}(\bar{\nu},\bar{q},\bar{\mu})(\delta\nu,\delta q)
	-c_1\abs{\delta\nu}^2 - c_1\abs{\delta\nu}\norm{\delta q}_{L^2(I\times\omega)}.
	\]
	Since $\partial_{\nu}\mathcal{L}(\bar{\nu},\bar{q},\bar{\mu}) = 0$ and using \cref{prop:langrange_lowerbound_quadratic_l1}, this further implies
	\[
	\nu - \bar{\nu} \geq \frac{\bar{\nu}}{2}\Psi^{-1}\left(c_0\norm{\delta{q}}_{L^1(I\times\omega)}\right)\norm{\delta{q}}_{L^1(I\times\omega)}-c_1\abs{\delta\nu}^2 - c_1\abs{\delta\nu}\norm{\delta q}_{L^2(I\times\omega)}.
	\]
	Clearly, we have
	\[
	\norm{\delta q}_{L^2(I\times\omega)} \leq \norm{\delta q}^{1/2}_{L^\infty(I\times\omega)}\norm{\delta q}^{1/2}_{L^1(I\times\omega)} \leq \abs{q_b-q_a}^{1/2}\norm{\delta q}^{1/2}_{L^1(I\times\omega)}.
	\]
	Employing \cref{prop:young_inequality_Psiinverse} for
	some $\varepsilon > 0$ to determined later, we obtain
	\begin{equation*}
	\abs{\delta\nu}\norm{\delta q}^{1/2}_{L^1(I\times\omega)}
	\leq \varepsilon\Psi^{-1}\left(c_0 \norm{\delta q}_{L^1(I\times\omega)}\right)\norm{\delta q}_{L^1(I\times\omega)} + F(\abs{\delta\nu}).
	\end{equation*}
	Taking $\varepsilon = \left(4c_1\abs{q_b-q_a}^{1/2}\right)^{-1}\bar{\nu}$,
	we obtain
	\[
	\delta\nu + F(\abs{\delta\nu}) + c_1\abs{\delta\nu}^2
	\geq \frac{\bar{\nu}}{4}\Psi^{-1}\left(c_0 \norm{q-\bar{q}}_{L^1(I\times\omega)}\right)\norm{q-\bar{q}}_{L^1(I\times\omega)}.
	\]
	Finally, using that $\lim_{y \to 0} F(y)/y = 0$, 
	we deduce
	\[
	F(\abs{\delta\nu}) + c_1\abs{\delta\nu}^2 \leq \frac{1}{2}\abs{\delta\nu},\quad \abs{\delta\nu} \leq \delta,
	\]
	for $\delta > 0$ sufficiently small, concluding the proof.
\end{proof}	
\begin{remark}
	\label{remark:growth_condition_Psi_kappa_epsilon}
	For the special case $\Psi(\varepsilon) = C\varepsilon^\kappa$,
	in \cref{prop:langrange_lowerbound_quadratic_l1} we obtain
	\[
	\partial_{q}\mathcal{L}(\bar{\nu},\bar{q},\bar{\mu})(q-\bar{q}) \geq c\norm{q-\bar{q}}^{1+1/\kappa}_{L^1(I\times\omega)}.
	\]
	Moreover, the growth condition from \cref{thm:adjoint_based_ssc_bang_bang} reads as follows:
	There are $\delta > 0$ and $c > 0$ such that
	\[
	c\norm{q-\bar{q}}^{1+1/\kappa}_{L^1(I\times\omega)} \leq \nu - \bar{\nu}
	\]
	for all admissible $(\nu, q) \in \Rplus\times\Qad(0,1)$
	with $\abs{\nu - \bar{\nu}} \leq \delta$.
\end{remark}

\section{A priori discretization error estimates}
\label{sec:robust_error_estimates_bb}
The aim of this section is the derivation of discretization error estimates for bang-bang controls based
on the different conditions of the preceding section. 
We consider the same assumptions concerning the temporal and spatial discretization of 
the partial differential equation as in~\cite{Bonifacius2017a},
which will be summarized in the following for the convenience of the reader.
Let 
\begin{equation*}
[0, 1] = \{0\}\cup I_1\cup I_2 \cup\ldots \cup I_M
\end{equation*}
be a partitioning of the reference time interval $[0, 1]$ with
disjoint subintervals $I_m = (t_{m-1},t_m]$ of size $k_m$ 
defined by the time points 
\begin{equation*}
0 = t_0 < t_1 < \ldots < t_{M-1} < t_M = 1\mbox{.}
\end{equation*}
Moreover, let $k$ denote the time discretization parameter 
defined as the piecewise constant function $k|_{I_m} = k_m$ for all $m = 1,2,\ldots,M$.
We also set $k = \max k_m$ the maximal time step size.
The temporal mesh is assumed to be regular in the sense of \cite[Section~3.1]{Meidner2011a}.

Concerning the spatial discretization,
let $\mathcal{T}_h = \set{K}$ be a mesh consisting of 
triangular or tetrahedral cells $K$ that form a non-overlapping cover of the domain $\Omega$.
The corresponding spatial discretization parameter 
$h$ is the cellwise constant function $h|_K = h_K$, 
where $h_K$ is the diameter of the cell $K$.
In addition, we set $h = \max h_K$.
Let $V_h \subset H^1_0$ denote the subspace of continuous and cellwise linear functions.
We assume that the $L^2$-projection onto $V_h$, denoted by
$\ProjH\colon L^2 \rightarrow V_h$, is stable in $H^1$.
This is satisfied if the mesh is globally quasi-uniform,
but weaker conditions are known; see \cite{Bramble2002}.
We construct the space-time finite element space in a standard way by
setting
\begin{equation*}
\Xkh = \left\{v_{kh} \in L^2(I; V_h) \constraintSet v_{kh}|_{I_m} \in \mathcal P_0(I_m; V_h)\mbox{,~}m =1,2,\ldots,M \right\}\mbox{,}
\end{equation*}
where $\mathcal P_0(I_m; V_h)$ is the space of constant functions on the time interval~$I_m$ with values in~$V_h$.	
Moreover, for $\varphi_k\in \Xkh$ we set $\varphi_{k,m} \ldef \varphi_k(t_m)$ with $m=1,2,\ldots,M,$ as well as
$[\varphi_k]_m \ldef \varphi_{k,m+1}-\varphi_{k,m}$ for $m=1,2,\ldots,M-1$.

In order to introduce the discrete version to the state equation, 
let the trilinear form $\B \colon \R\times\Xkh\times\Xkh \rightarrow \R$ be defined as
\begin{multline*}
\B(\nu, u_{kh}, \varphi_{kh}) \ldef \sum_{m=1}^M \pair{\partial_t u_{kh}, \varphi_{kh}}_{L^2(I_m; L^2)}\\
+ \nu\inner{\nabla u_{kh}, \nabla\varphi_{kh}}_{L^2(I; L^2)} + \sum_{m=2}^M([u_{kh}]_{m-1}, \varphi_{kh,m}) + \inner{u_{kh,1},\varphi_{kh,1}}\mbox{.}
\end{multline*}
Given $\nu \in \Rplus$ and $q \in Q(0,1)$ the discrete state equation reads as follows: Find a state $u_{kh} \in \Xkh$ satisfying
\begin{equation}
\B(\nu, u_{kh},\varphi_{kh}) = \nu\inner{\ControlOp q,\varphi_{kh}}_{L^2(I; L^2)} + \inner{u_0,\varphi_{kh,1}}_{L^2}\quad\text{for all }\varphi_{kh}\in \Xkh\mbox{.}\label{eq:stateEquationDiscrete}
\end{equation}
We also introduce the discrete Laplace operator $-\Lap_h\colon V_h\to V_h$ by
\begin{equation*}
-\inner{\Lap_h u_h, \varphi_h}_{L^2} = \inner{\nabla u_h, \nabla \varphi_h}_{L^2}\mbox{,}\quad \varphi_h \in V_h\mbox{.}
\end{equation*}

Next, we introduce a discrete control variable.
To consider different discretization schemes in one consistent notation, 
we introduce the operator $\ProjDiscControl$
onto the possibly discrete control space $\Qsigma(0,1) \subset \Lcontrol{2}$,
where $\sigma$ is abstract parameter for the control discretization. 
To simplify the discussion, we assume that in the case of a distributed control
a subset denoted $\mathcal{T}^\omega_h$ of the mesh $\mathcal{T}_h$ is a non-overlapping cover of~$\omega$. 
Furthermore, we suppose that the optimal control $\bar{q}$ satisfies
\begin{equation}
\norm{B\left(\bar{q} - \ProjDiscControl\bar{q}\right)}_{L^2(I; H^{-1})}  \leq \sigma(k,h)\mbox{,}\label{eq:estimate_projection_discrete_controls_bb2}
\end{equation}
where $\sigma(k,h) \rightarrow 0$ as $k, h \rightarrow 0$
and $\ProjDiscControl\Qad(0,1) \subset \Qad(0,1)$. 
We also simply write $\ProjDiscControl (\nu,q) = (\nu, \ProjDiscControl q)$ using the same symbol and define $\Qsigmaad(0,1) = \Qsigma(0,1)\cap\Qad(0,1)$.
Concrete discretization schemes for the control will be discussed at the end of this section.

We define the discretized optimal control problem 
corresponding to \eqref{Pt} by
\begin{equation}\label{PkhAlpha}\tag{\mbox{$\hat{P}_{kh}$}}
\begin{aligned}
\mbox{Minimize~} \nu_{kh} \quad\mbox{subject to}\quad \nu_{kh}&\in\Rplus\mbox{,~} 
q_{kh} \in \Qsigmaad(0,1)\mbox{,}\\
g_{kh}(\nu_{kh},q_{kh}) &\leq 0\mbox{,}
\end{aligned}
\end{equation}
where $g_{kh}(\nu_{kh},q_{kh}) = G(i_1S_{kh}(\nu_{kh},q_{kh}))$
and $S_{kh}$ denotes the control-to-state mapping for the discrete state equation~\eqref{eq:stateEquationDiscrete}.
In the following, $\{(k,h)\}$ is always a
sequence of positive mesh sizes converging to zero.		

\subsection{Error estimates for the terminal times}
\label{subsec:discrete_problem}

Similar as in
\cite{Bonifacius2017a}	
we construct two auxiliary sequences: First, we construct $\{(\nu_\gamma,q_\gamma)\}_{\gamma > 0}$ converging to $(\bar{\nu},\bar{q})$ as $\gamma\to 0$ that is feasible for~\eqref{PkhAlpha}. 
In particular, this ensures existence of a solution to the discrete problem.
Moreover, we obtain a first convergence result without rates.
Thereafter, we construct another sequence $\{(\nu_\tau,q_\tau)\}_{\tau > 0}$ 
converging to $(\bar{\nu}_{kh},\bar{q}_{kh})$ as $\tau\to 0$ 
that is feasible for~\eqref{Pt}. 
Since the solution operator to the state equation is continuous for right-hand sides from $L^2(I; H^{-1})$ into $W(0,1) \embedding C([0,1]; L^2)$, we may use \eqref{eq:estimate_projection_discrete_controls_bb2} for all estimates concerning the state or the linearized state.	
Note that all sequences constructed in~\cite{Bonifacius2017a}
are independent of the cost parameter $\alpha$. 

For the error estimates, we require the
following stability and discretization error estimates
that are essentially based on \cite[Propositions~4.4 and~4.6]{Bonifacius2017a}.
\begin{proposition}
	\label{prop:stability_discretization_g}
	Let $0 < \nu_{\min} < \nu_{\max}$ be fixed. 
	Then for all $\nu_{\min} \leq \nu \leq \nu_{\max}$ and $q \in \Qad(0,1)$ we have
	\begin{equation}
	\abs{\partial_{\nu\nu} g_{kh}(\nu, q)} \leq c,\label{eq:stabilityEstimateGkh2_nu}
	\end{equation}		
	where $c > 0$ is a constant independent of $\nu$, $q$, $k$, and $h$.
	Moreover, 
	\begin{align}
	\abs{g(\nu,q) - g_{kh}(\nu,q)} &\leq c\abs{\log k} (k + h^2)\left(\norm{\ControlOp q}_{L^\infty(I; L^2)} + \norm{u_0}_{L^2}\right)\mbox{,}\label{eq:errorEstimatesG1}\\
	\abs{\partial_{\nu}g(\nu,q) - \partial_{\nu}g_{kh}(\nu,q)} &\leq c \abs{\log k} (k + h^2)\left(\norm{\ControlOp q}_{L^\infty(I; L^2)} + \norm{u_0}_{H^1}\right)\mbox{,}\label{eq:errorEstimatesG2}
	\end{align}
	where $c > 0$ is a constant independent of $\nu$, $q$, $k$, and $h$.
\end{proposition}

\begin{proposition}\label{prop:auxiliarySequenceGamma_bb}
	Let $(\bar{\nu},\bar{q})$ be a globally optimal control of problem~\eqref{Pt}. There exists a sequence $\{(\nu_\gamma,q_\gamma)\}_{\gamma > 0}$ of controls with $\gamma = \gamma(k,h)$ that are feasible for~\eqref{PkhAlpha} for $k$ and $h$ sufficiently small. Moreover, we have the estimate
	\begin{equation*}
	\abs{\nu_{\gamma}-\bar{\nu}}
	\leq c\left(\sigma(k,h) + \abs{\log k} (k + h^2)\right)\mbox{.}
	\end{equation*}
\end{proposition}
\begin{proof}
	The sequence can be constructed as in
	\cite[Proposition~4.7]{Bonifacius2017a}.	
	 We give the proof for convenience and
	 abbreviate $\bar{\chi} = (\bar{\nu},\bar{q})$. For $\gamma > 0$
	 to be determined in the course of the proof we set
	 \begin{equation*}
	 \chi_\gamma \ldef \ProjDiscControl \breve{\chi}^\gamma
	 = (\bar{\nu}+\gamma,\ProjDiscControl \bar{q}) \mbox{.}
	 \end{equation*}
	 Using Taylor expansion of $g_{kh}$ at $\ProjDiscControl\bar{\chi}$ we
	 find for some \(\chi_\zeta\) that
	 \begin{equation}
	 g_{kh}(\chi_\gamma)
	 = g_{kh}(\ProjDiscControl\bar{\chi})
	 + \gamma \, g_{kh}'(\ProjDiscControl\bar{\chi})(1,0)
	 + \frac{\gamma^2}{2} \, g_{kh}''(\chi_\zeta)[1,0]^2\mbox{.}\label{eq:auxiliarySequenceGamma_bb_P2}
	 \end{equation}
	 Using the triangle inequality we estimate the first term of~\eqref{eq:auxiliarySequenceGamma_bb_P2} by
	 \begin{align}
	 g_{kh}(\ProjDiscControl\bar{\chi})
	 &\leq g(\bar{\chi}) + \abs{g_{kh}(\ProjDiscControl\bar{\chi}) - g(\ProjDiscControl\bar{\chi})}
	 + \abs{g(\ProjDiscControl\bar{\chi}) - g(\bar{\chi})}\nonumber \\
	 &\leq c\abs{\log k} (k + h^2)
	 + c\norm{\ControlOp\left(\ProjDiscControl\bar{q} - \bar{q}\right)}_{L^2(I; H^{-1})}\nonumber \\
	 &\leq c_1(\abs{\log k} (k + h^2) + \sigma(k,h))
	 \rdef \delta_1(k,h)\label{eq:auxiliarySequenceGamma_bb_P5}
	 \end{align}
	 with \eqref{eq:errorEstimatesG1} and Lipschitz continuity of $g$.
	 For the second term of~\eqref{eq:auxiliarySequenceGamma_bb_P2}, we estimate
	 \begin{equation}
	 g_{kh}'(\ProjDiscControl\bar{\chi})(1,0)
	 \leq g'(\bar{\chi})(1,0) + c_2\left(\abs{\log k} (k + h^2) + \sigma(k,h)\right)
	 \leq -\bar{\eta} + \delta_2(k,h)\mbox{,}\label{eq:auxiliarySequenceGamma_bb_P6}
	 \end{equation}
	 using \cref{assumption:linearized_slater}, and \(g'(\bar{\chi})(1,0)
	 = \partial_\nu g(\bar{\chi})\).
	 Finally, for the third term, we obtain
	 \(g_{kh}''(\chi_\zeta)[\gamma,0]^2 \leq c_3 \gamma^2\) due to \eqref{eq:stabilityEstimateGkh2_nu}.
	 Collecting the estimates, we have
	 \begin{equation*}
	 g_{kh}(\chi_{\gamma}) \leq \delta_1(k,h)
	 - \gamma\left(\bar{\eta} - \delta_2(k,h) - c_3 \gamma\right)\mbox{.}
	 \end{equation*}
	 Note that the first component of $\chi_\gamma$ is bounded below by
	 $\bar{\nu}$ and bounded above by $\bar{\nu}+1$, so that all constants of
	 the error and stability estimates used above can be chosen to be
	 independent of $\chi_\gamma$. Taking 
	 \begin{equation*}
	 \gamma = \frac{3 \delta_1(k,h)}{\bar{\eta}} \leq \frac{\bar{\eta}}{3c_3}
	 \quad\text{and}\quad
	 \delta_2(k,h) \leq \frac{\bar{\eta}}{3}
	 \end{equation*}
	 for $k, h$ sufficiently small, we obtain $g_{kh}(\chi_\gamma) \leq 0$. From the 
	 definition of $\gamma$ we further deduce $\gamma = \gamma(k,h) =
	 \mathcal{O}(\sigma(k,h) + \abs{\log k} (k + h^2))$.
\end{proof}

In particular, \cref{prop:auxiliarySequenceGamma_bb} implies existence of feasible points for the discrete problem~\eqref{PkhAlpha},
which in turn guarantees existence of an optimal solution to the discrete problem.
Even better, we obtain a first convergence result.

\begin{lemma}\label{prop:convergence_Pkh}
	Let $(\bar{\nu},\bar{q})$ be an optimal solution of problem~\eqref{Pt}
	such that \cref{assumption:linearized_slater} holds.		
	For $k$ and $h$ sufficiently small, 
	the discrete problem~\eqref{PkhAlpha}
	has an optimal solution
	$(\bar{\nu}_{kh},\bar{q}_{kh}) \in \Rplus\times\Qsigmaad(0,1)$.
	Moreover, $\bar{\nu}_{kh} \to \bar{\nu}$ and every weak limit
	of $(\bar{q}_{kh})_{k,h > 0}$ is optimal for~\eqref{Pt}.
\end{lemma}
\begin{proof}
	Existence of solutions follows by standard arguments, since the set of admissible controls is nonempty according to \cref{prop:auxiliarySequenceGamma_bb}. 
	Moreover, using optimality of $(\bar{\nu}_{kh},\bar{q}_{kh})$, 
	feasibility of $(\nu_\gamma,q_\gamma)$, and $0 \leq \gamma \leq 1$, we observe
	\[
	0 \leq \bar{\nu}_{kh} \leq \nu_\gamma = \bar{\nu} + \gamma \leq \bar{\nu} + 1.
	\]
	Hence, $(\bar{\nu}_{kh},\bar{q}_{kh})$ is uniformly bounded.
	Thus, there exists a subsequence denoted in the same way
	such that $\bar{\nu}_{kh} \to \nu^*$ and
	$q_{kh} \rightharpoonup q^*$ in $L^s(I\times\omega)$ with $q^* \in \Qad(0,1)$ and some $s > 2$.		
	Feasibility of $(\bar{\nu}_{kh}, \bar{q}_{kh})$ for~\eqref{PkhAlpha} further yields
	\begin{align*}
	g(\nu^*, q^*) &\leq g_{kh}(\bar{\nu}_{kh}, \bar{q}_{kh}) 
	+ \abs{g(\bar{\nu}_{kh}, \bar{q}_{kh})-g_{kh}(\bar{\nu}_{kh}, \bar{q}_{kh})}
	+ \abs{g(\nu^*, q^*)-g(\bar{\nu}_{kh}, \bar{q}_{kh})}\\
	&\leq c \norm{i_1S(\bar{\nu}_{kh}, \bar{q}_{kh})-i_1S_{kh}(\bar{\nu}_{kh}, \bar{q}_{kh})}_{L^2} + c \norm{i_1S(\nu^*, q^*)-i_1S(\bar{\nu}_{kh}, \bar{q}_{kh})}_{L^2},
	\end{align*}
	where we have used Lipschitz continuity of $G$ on bounded sets in $L^2$.
	Going to the limit $k,h \to 0$, employing the convergence result~\cref{prop:uniformConvergenceStateTeminalValue}
	as well as complete continuity \cref{prop:complete_continuity_control_to_terminal_obs}, 
	we deduce that $g(\nu^*, q^*) \leq 0$. In particular, $\bar{\nu} \leq \nu^*$.
	
	Optimality of $(\bar{\nu}_{kh}, \bar{q}_{kh})$
	and feasibility of $(\nu_\gamma,q_\gamma)$ from \cref{prop:auxiliarySequenceTau_bb} for~\eqref{PkhAlpha},
	leads to
	\[
	\nu^* = \lim_{k,h \to 0} \bar{\nu}_{kh} 
	\leq \lim_{k,h \to 0} \nu_\gamma = \lim_{k,h \to 0} \bar{\nu} + c\left(\sigma(k,h) + \abs{\log k} (k + h^2)\right) = \bar{\nu}.
	\]
	Hence, $\bar{\nu} = \nu^*$ and $(\bar{\nu},q^*)$ is also optimal.
	Moreover, as the limit $\bar{\nu}$ is independent of 
	the concretely chosen subsequence, the whole sequence converges.
\end{proof}
In addition, \cref{prop:convergence_Pkh} implies that the sequence $\bar{\nu}_{kh}$ 
is uniformly bounded away from zero.	
Hence, the constants in the following error estimates can be chosen to be independent of $\bar{\nu}_{kh}$;
cf.\ \cref{prop:stability_discretization_g,lemma:errorEstimatesStateEquationLinf}.

As the next step towards error estimates,
we verify that the linearized Slater condition holds at $(\bar{\nu}_{kh},\bar{q}_{kh})$ for the discrete problem.
From now on we assume uniqueness of the optimal solution; 
recall \cref{assumption:control_bang_bang_measure_condition,prop:control_bangbang_unique}
for a sufficient condition.
\begin{proposition}\label{prop:slaterPointDiscrete_bb}
	Let $(\bar{\nu},\bar{q})$ be the unique optimal solution of~\eqref{Pt}
	such that \cref{assumption:linearized_slater} holds.
	Moreover, let $(\bar{\nu}_{kh},\bar{q}_{kh}) \in \Rplus\times\Qsigmaad(0,1)$
	be an optimal solution of \eqref{PkhAlpha}.
	For $k$ and $h$ sufficiently small we have
	\begin{equation*}
	\partial_\nu g_{kh}(\bar{\nu}_{kh},\bar{q}_{kh}) \leq -\bar{\eta}/2 < 0\mbox{.}
	\end{equation*}
\end{proposition}
\begin{proof}
	We use the representation of $g'$, i.e.\ $\partial_\nu g(\nu,q) =
	\int_{0}^{1}\pair{Bq+ \Lap u, z}$ from \eqref{eq:terminal_constraint_adjoint}.
	Then, the discretization error estimate~\eqref{eq:errorEstimatesG2} implies
	\begin{multline*}
	\partial_\nu g_{kh}(\bar{\nu}_{kh},\bar{q}_{kh})
	\leq \partial_\nu g(\bar{\nu},\bar{q}) + c\abs{\log k}\left(k + h^2\right)\\
	+ \abs*{\int_{0}^{1}\pair{B\bar{q}_{kh} + \Lap u(\bar{\nu}_{kh},\bar{q}_{kh}), z(\bar{\nu}_{kh},\bar{q}_{kh})} - \int_{0}^{1}\pair{B\bar{q} + \Lap\bar{u}, z(\bar{\nu},\bar{q})}},
	\end{multline*}
	where $z(\bar{\nu},\bar{q})$ and $z(\bar{\nu}_{kh},\bar{q}_{kh})$ denote the adjoint states
	with terminal values $\bar{u}(1)-u_d$ and $i_1S(\bar{\nu}_{kh},\bar{q}_{kh}) - u_d$
	and time transformations $\bar{\nu}$ and $\bar{\nu}_{kh}$.
	The convergence result~\cref{prop:convergence_Pkh} and
	complete continuity of the control-to-observation mapping, see \cref{prop:complete_continuity_control_to_terminal_obs},
	imply $z(\bar{\nu}_{kh},\bar{q}_{kh}) \to \bar{z}$ in $W(0,1)$.
	Hence, the result follows from the linearized Slater condition~\eqref{eq:definition_linearized_slater_condition}.
\end{proof}

\begin{proposition}\label{prop:auxiliarySequenceTau_bb}
	Let $k$ and $h$ be sufficiently small. Moreover, let $(\bar{\nu},\bar{q})$ be the unique optimal solution of~\eqref{Pt} and let $(\bar{\nu}_{kh},\bar{q}_{kh})$ be an optimal control of~\eqref{PkhAlpha}. Then there exists a sequence $(\nu_\tau)_{\tau > 0}$ such that
	$(\nu_\tau,\bar{q}_{kh})$ are feasible for~\eqref{Pt} and
	\begin{equation*}
	\abs{\nu_{\tau}-\bar{\nu}_{kh}}
	\leq c\abs{\log k} (k + h^2)\mbox{.}
	\end{equation*}
\end{proposition}
\begin{proof}
	Proceeding as in \cite[Proposition~4.10]{Bonifacius2017a} we set
	\[
	\chi_\tau = (\nu_\tau, q_\tau) = (\bar{\nu}_{kh} + \tau, \bar{q}_{kh}).
	\]
	for some \(\tau \in (0,1]\) to be determined later.
	Now, the proof is along the lines of the one of \cref{prop:auxiliarySequenceGamma_bb},
	interchanging the roles of \(\bar{\chi} = (\bar{\nu}, \bar{q})\) and \(\bar{\chi}_{kh} = (\bar{\nu}_{kh},\bar{q}_{kh})\) and \(g\) and \(g_{kh}\).
\end{proof}
\begin{lemma}
	\label{lemma:robust_error_estimate_times}
	Let $(\bar{\nu},\bar{q})$ be the unique optimal solution of problem~\eqref{Pt}
	such that \cref{assumption:linearized_slater} holds.
	Moreover, let $(\bar{\nu}_{kh},\bar{q}_{kh}) \in \Rplus\times\Qsigmaad(0,1)$
	be an optimal solution of \eqref{PkhAlpha}.
	Then, for $k$ and $h$ sufficiently small, we have
	\[
	\abs{\bar{\nu}-\bar{\nu}_{kh}} \leq c \left(\sigma(k,h) + \abs{\log k} (k + h^2)\right)\mbox{,}
	\]
	where $c > 0$ is independent of $k$ and $h$.
	Moreover, there exists a unique Lagrange multiplier $\bar{\mu}_{kh} = \bar{\mu}_{kh}(\bar{q}_{kh}) > 0$ such that the optimality system is satisfied
	\begin{align}
	\int_0^1 1 + \pair{\ControlOp\bar{q}_{kh}(t) + \Lap_h\bar{u}_{kh}(t), \bar{z}_{kh}(t)} \D{t} &= 0\mbox{,}\label{eq:opt_cond_hamiltonianConstant_kh_alpha}\\
	\int_0^{1} \pair{\ControlOp^*\bar{z}_{kh}(t), q(t) - \bar{q}_{kh}(t)}\D{t} &\geq 0\quad \text{for all~} q \in\Qsigmaad(0,1)\mbox{,}\label{eq:opt_cond_variationalInequality_kh_alpha}\\
	G(\bar{u}_{kh}(1)) &= 0\mbox{,}\label{eq:opt_cond_feasibility_kh_alpha}
	\end{align}
	where \(\bar{u}_{kh} = S(\bar{\nu}_{kh},\bar{q}_{kh})\) and $\bar{z}_{kh} \in \Xkh$ is the solution to the discrete adjoint equation
	\begin{equation*}
	\B(\bar{\nu}_{kh}, \varphi_{kh}, \bar{z}_{kh}) = \bar{\mu}_{kh}(\bar{u}_{kh}(1)-u_d, \varphi_{kh}(1))\mbox{,}\quad \varphi_{kh} \in \Xkh\mbox{.}
	\end{equation*}	
\end{lemma}
\begin{proof}
	Because the pair $(\nu_\tau, \bar{q}_{kh})$ is feasible for~\eqref{Pt},
	we have
	\begin{equation*}
	0 \leq \nu_\tau - \bar{\nu}
	= \nu_\tau - \bar{\nu}_{kh} + \bar{\nu}_{kh} - \nu_\gamma + \nu_\gamma - \bar{\nu}
	\leq \nu_\tau - \bar{\nu}_{kh} + \nu_\gamma - \bar{\nu}\mbox{,}
	\end{equation*}
	where the last inequality follows from optimality of the pair $(\bar{\nu}_{kh}, \bar{q}_{kh})$ for~\eqref{PkhAlpha} and feasibility of $(\nu_\gamma,q_\gamma)$ for~\eqref{PkhAlpha}.
	Hence, 
	\begin{align*}
	\abs{\bar{\nu}_{kh}-\bar{\nu}}
	&\leq \abs{\bar{\nu}_{kh}-\nu_\tau} + \nu_\tau-\bar{\nu}
	\leq 2\abs{\bar{\nu}_{kh}-\nu_\tau} + \abs{\nu_\gamma - \bar{\nu}}\\
	&\leq c\left(\sigma(k,h) + \abs{\log k} (k + h^2)\right).
	\end{align*}
	where we have used \cref{prop:auxiliarySequenceGamma_bb,prop:auxiliarySequenceTau_bb}
	in the last step.
	Finally, the linearized Slater condition due to \cref{prop:slaterPointDiscrete_bb} 
	yields the optimality conditions in qualified form as stated above.
\end{proof}
\begin{remark}\label{remark:uniqueness_multiplier}
	For each tuple $(\bar{\nu}_{kh},\bar{q}_{kh}) \in \Rplus\times\Qsigmaad(0,1)$, 
	there exists a unique Lagrange multiplier $\bar{\mu}_{kh}$.
	However, as the discrete control is not guaranteed to be unique,
	there might be different multipliers.
	Nevertheless, we can prove the a priori bound
	$\bar{\mu}_{kh} \leq 2/\bar{\eta}$ for $k$ and $h$ sufficiently small
	using the optimality conditions for~\eqref{PkhAlpha} and \cref{prop:slaterPointDiscrete_bb}.
\end{remark}

\subsection{Convergence of controls}
Next, we prove convergence of the control variable
based on the growth condition~\eqref{eq:quadratic_growth_bb}.
\begin{theorem}\label{lemma:robust_error_estimate_bb_suboptimal}
	Let $(\bar{\nu},\bar{q})$ be the global solution to~\eqref{Pt}
	such that~\cref{assumption:linearized_slater,assumption:control_bang_bang_measure_condition} hold.	
	Moreover, let $(\bar{\nu}_{kh},\bar{q}_{kh}) \in \Rplus\times\Qsigmaad(0,1)$
	be an optimal solution of \eqref{PkhAlpha}.
	Then, we have $\bar{q}_{kh} \to \bar{q}$ in
	$L^1(I\times\Omega)$ and for $k$ and $h$ sufficiently small it holds
	\begin{equation}\label{eq:robust_error_estimate_times}
	\abs{\bar{\nu}-\bar{\nu}_{kh}} \leq c \left(\sigma(k,h) + \abs{\log k} (k + h^2)\right)\mbox{.}
	\end{equation}
\end{theorem}
\begin{proof}		
	Let $\{(\bar{\nu}_{kh},\bar{q}_{kh})\}$ be a sequence of globally optimal solutions to~\eqref{PkhAlpha} that is guaranteed due to \cref{prop:convergence_Pkh}.
	The error estimate for the optimal times~\eqref{eq:robust_error_estimate_times}
	is the assertion of \cref{lemma:robust_error_estimate_times}.
	Since
	\[
	\abs{\nu_{\tau}-\bar{\nu}} \leq \abs{\nu_{\tau}-\bar{\nu}_{kh}} + \abs{\bar{\nu}_{kh}-\bar{\nu}} \to 0,
	\]
	and because the pair $(\nu_\tau, \bar{q}_{kh})$ is feasible for~\eqref{Pt},
	we may use the growth condition 
	from \cref{thm:adjoint_based_ssc_bang_bang} to deduce
	\begin{equation}
	\label{eq:subotimal_estimate_controls}
	\frac{\bar{\nu}}{4}\Psi^{-1}\left(c_0 \norm{\bar{q}_{kh}-\bar{q}}_{L^1(I\times\omega)}\right)\norm{\bar{q}_{kh}-\bar{q}}_{L^1(I\times\omega)} 
	\leq \nu_\tau - \bar{\nu}.
	\end{equation}
	Strict monotonicity and continuity of $\Psi^{-1}$
	finally imply $\bar{q}_{kh} \to \bar{q}$ in $L^1(I\times\Omega)$.		
\end{proof}
\begin{remark}
	If $\Psi(\varepsilon) = C\varepsilon^\kappa$,
	then in view of \cref{remark:growth_condition_Psi_kappa_epsilon} we obtain
	from~\eqref{eq:subotimal_estimate_controls}
	with similar arguments as in the proof of \cref{lemma:robust_error_estimate_times}
	the sub-optimal estimate
	\[
	c \norm{\bar{q}_{kh}-\bar{q}}_{L^1(I\times\omega)}^{1+1/\kappa} \leq \nu_\tau - \bar{\nu} \leq c \left(\sigma(k,h) + \abs{\log k} (k + h^2)\right).
	\]
	An improved estimate will be derived in the next section.
\end{remark}

\subsection{Improved error estimates for controls}
\label{subsec:improved_estimates}

Under certain conditions we will eventually provide an improved error estimate that is directly based on the structural condition~\eqref{eq:assumption_structure_adjoint_ssc_bb}.
The required improved regularity in case of a distributed control
is satisfied, if, e.g., $u_0 \in \dom{L^p}{-\Lap}$ with $p > d/2$, 
where we recall that $d$ denotes the spatial dimension.
\begin{proposition}\label{prop:robust_estimate_bb}
	Adopt the assumptions of \cref{lemma:robust_error_estimate_bb_suboptimal}.
	Moreover, we assume that $\ProjDiscControl$ is the orthogonal projection onto $\Qsigma(0,1)$ in $L^2(I\times\omega)$.
	In case of a distributed control, suppose in addition that 
	$u_0 \in (L^p,\dom{L^p}{-\Lap})_{1-1/s,s}$ for $s, p \in (1,\infty)$ such that $d/(2p) + 1/s < 1$.
	There is a constant $c > 0$ independent of $k$, $h$, $\bar{\nu}_{kh}$, and $\bar{q}_{kh}$ such that
	\begin{multline*}
	\Psi^{-1}\left(c_0\norm{\bar{q}-\bar{q}_{kh}}_{L^1(I\times\omega)}\right) \\
	\leq c\Big(\abs{\bar{\nu}-\bar{\nu}_{kh}} 
	+ \norm{(\Id - \ProjDiscControl)\ControlOp^*\hat{z}_{kh}}_{L^\infty(I\times\omega)}
	+ \norm{\ControlOp^*\left(\hat{z}_{kh} - \hat{z})\right)}_{L^\infty(I\times\omega)}\Big),
	\end{multline*}
	where $\hat{z} \in W(0,1)$ solves
	\[
	-\partial_t \hat{z} - \bar{\nu}_{kh}\Lap\hat{z} = 0, 
	\quad \hat{z}(1) = \bar{\mu}\left(\bar{u}_{kh}(1) - u_d\right),
	\]
	and $\hat{z}_{kh} = (\bar{\mu}/\bar{\mu}_{kh}) \bar{z}_{kh} \in \Xkh$ solves
	\begin{equation}
	\label{eq:robust_estimate_bb_aux_adjoint_discrete}
	\B(\bar{\nu}_{kh}, \varphi_{kh}, \hat{z}_{kh}) = \bar{\mu}\inner{\bar{u}_{kh}(1) - u_d, \varphi_{kh}(1)},
	\quad \varphi_{kh} \in \Xkh.
	\end{equation}
\end{proposition}	

\begin{proof}	
	As in the proof of \cite[Theorem~31]{vonDaniels2017b},
	in~\eqref{eq:langrange_lowerbound_quadratic_l1}
	we set $q = \bar{q}_{kh}$ to obtain
	\begin{equation}\label{eq:robust_estimate_bb_P0}
	\frac{\bar{\nu}}{2}\Psi^{-1}\left(c_0\norm{\bar{q}-\bar{q}_{kh}}_{L^1(I\times\omega)}\right)
	\norm{\bar{q}-\bar{q}_{kh}}_{L^1(I\times\omega)} \leq -\int_{0}^{1}\inner{\ControlOp^*\bar{z}, \bar{q}-\bar{q}_{kh}}_{L^2(\omega)}\D{t}.
	\end{equation}
	The optimality condition~\eqref{eq:opt_cond_variationalInequality_kh_alpha} with $q = \ProjDiscControl\bar{q}$ 
	multiplied by $\bar{\mu}/\bar{\mu}_{kh} > 0$ reads		
	\begin{equation}\label{eq:robust_estimate_bb_P1}
	0 \leq \int_{0}^{1}\inner{\ControlOp^*\hat{z}_{kh}, \ProjDiscControl\bar{q}-\bar{q}_{kh}}_{L^2(\omega)}\D{t},
	\end{equation}
	where $\hat{z}_{kh} = (\bar{\mu}/\bar{\mu}_{kh})\bar{z}_{kh}$,
	i.e.\ $\hat{z}_{kh}$ fulfills the same discrete adjoint equation as $\bar{z}_{kh}$
	but with multiplier $\bar{\mu}$ instead of $\bar{\mu}_{kh}$, as given in~\eqref{eq:robust_estimate_bb_aux_adjoint_discrete}.
	Summation of~\eqref{eq:robust_estimate_bb_P0} and~\eqref{eq:robust_estimate_bb_P1} implies
	\begin{align}
	\frac{\bar{\nu}}{2}\Psi^{-1}&\left(c_0\norm{\bar{q}-\bar{q}_{kh}}_{L^1(I\times\omega)}\right)
	\norm{\bar{q}-\bar{q}_{kh}}_{L^1(I\times\omega)}\nonumber\\
	&\leq \int_{0}^{1}\inner{\ControlOp^*\left(\hat{z}_{kh} - \bar{z}\right), \bar{q}-\bar{q}_{kh}}_{L^2(\omega)}\D{t}
	-
	\int_{0}^{1}\inner{\ControlOp^*\hat{z}_{kh},\bar{q}-\bar{q}_{kh}}_{L^2(\omega)}\D{t}\nonumber\\ &\quad+\int_{0}^{1}\inner{\ControlOp^*\hat{z}_{kh},\ProjDiscControl\bar{q}-\bar{q}_{kh}}_{L^2(\omega)}\D{t}\nonumber\\
	&= \int_{0}^{1}\inner{\ControlOp^*\left(\hat{z}_{kh} - \bar{z}\right), \bar{q}-\bar{q}_{kh}}_{L^2(\omega)}\D{t}
	+ \int_{0}^{1}\inner{\ControlOp^*\hat{z}_{kh},\ProjDiscControl\bar{q}-\bar{q}}_{L^2(\omega)}\D{t}.\label{eq:robust_estimate_bb_P2}
	\end{align}	
	Concerning the first term of
	the right-hand side of \eqref{eq:robust_estimate_bb_P2}, we have
	\begin{multline}		
	\int_{0}^{1}\inner{\ControlOp^*\left(\hat{z}_{kh} - \bar{z}\right), \bar{q}-\bar{q}_{kh}}_{L^2(\omega)}\D{t} 
	= \int_{0}^{1}\inner{\ControlOp^*\left(\hat{z}_{kh} - \hat{z}\right), \bar{q}-\bar{q}_{kh}}_{L^2(\omega)}\D{t}\\
	+ \int_{0}^{1}\inner{\ControlOp^*\left(\hat{z} - \widetilde{z}\right), \bar{q}-\bar{q}_{kh}}_{L^2(\omega)}\D{t}
	+ \int_{0}^{1}\inner{\ControlOp^*\left(\widetilde{z} - \bar{z}\right), \bar{q}-\bar{q}_{kh}}_{L^2(\omega)}\D{t},\label{eq:robust_estimate_bb_P8}
	\end{multline}
	where $\widetilde{z} = z(\bar{\nu},\bar{q}_{kh}) \in W(0,1)$ is an additional adjoint state solving 
	\[
	-\partial_t \widetilde{z} - \bar{\nu}\Lap \widetilde{z} = 0, 
	\quad \widetilde{z}(1) = \bar{\mu}\left(\widetilde{u}(1)- u_d\right),
	\quad \widetilde{u} = S(\bar{\nu},\bar{q}_{kh})
	\]
	Note that all adjoint states appearing above correspond to the same multiplier $\bar{\mu}$.
	For the first term on the right-hand side of~\eqref{eq:robust_estimate_bb_P8}, H\"older's inequality yields
	\begin{equation*}
	\int_{0}^{1}\inner{\ControlOp^*\left(\hat{z}_{kh} - \hat{z}\right), \bar{q}-\bar{q}_{kh}}_{L^2(\omega)}
	\leq \norm{\ControlOp^*\left(\hat{z}_{kh} - \hat{z}\right)}_{L^\infty(I\times\omega)} \norm{\bar{q}-\bar{q}_{kh}}_{L^1(I\times\omega)}.
	\end{equation*}
	The second term on the right-hand side of~\eqref{eq:robust_estimate_bb_P8} can be estimated using \cref{prop:estimate_transformations_LinfL2} 
	in case of purely time-dependent controls
	and \cref{prop:estimate_transformations_LinfLinf}
	in case of a distributed control as
	\[
	\int_{0}^{1}\inner{\ControlOp^*\left(\hat{z} - \widetilde{z}\right),\bar{q}-\bar{q}_{kh}}_{L^2(\omega)}
	\leq c\abs{\bar{\nu}_{kh} - \bar{\nu}} \norm{\bar{q}-\bar{q}_{kh}}_{L^1(I\times\omega)}.
	\]
	The third term on the right-hand side of~\eqref{eq:robust_estimate_bb_P8} 
	is less than or equal to zero, because of affine 
	linearity of $i_1S(\bar{\nu},q)$ with respect to $q$
	which implies
	\begin{multline*}
	\int_{0}^{1}\inner{\ControlOp^*\left(\widetilde{z} - \bar{z}\right), \bar{q}-\bar{q}_{kh}}_{L^2(\omega)}\D{t}\\
	= \bar{\mu}\inner{\left(i_1\left(\partial_t -
		\bar{\nu}\Lap\right)^{-1}\right)^*\left(\widetilde{u}(1) - \bar{u}(1)\right), \ControlOp(\bar{q}-\bar{q}_{kh})}_{L^2}
	= - \bar{\mu}\norm{\bar{u}(1) - \widetilde{u}(1)}_{L^2}^2.
	\end{multline*}	
	where $\left(\partial_t - \bar{\nu}\Lap\right)^{-1}$ denotes
	the solution operator to the linear heat-equation with homogeneous initial data.		
	Since $\ProjDiscControl$ is the $L^2(I\times\omega)$-projection onto $\Qsigma(0,1)$
	for the last term of the right-hand side of~\eqref{eq:robust_estimate_bb_P2} we obtain
	\begin{equation*}
	\int_{0}^{1}\inner{\ControlOp^*\hat{z}_{kh},\ProjDiscControl\bar{q}-\bar{q}}_{L^2(\omega)}\D{t}
	= \int_{0}^{1}\inner{(\Id - \ProjDiscControl)\ControlOp^*\hat{z}_{kh},\bar{q}_{kh}-\bar{q}}_{L^2(\omega)}\D{t}.
	\end{equation*}
	In summary, we arrive at
	\begin{multline*}
	\frac{\bar{\nu}}{2}\Psi^{-1}\left(c_0\norm{\bar{q}-\bar{q}_{kh}}_{L^1(I\times\omega)}\right)
	\norm{\bar{q}-\bar{q}_{kh}}_{L^1(I\times\omega)}
	\leq c\Big(\abs{\bar{\nu}_{kh} - \bar{\nu}}\\
	+ \norm{(\Id - \ProjDiscControl)\ControlOp^*\hat{z}_{kh}}_{L^\infty(I\times\omega)}
	+ \norm{\ControlOp^*\left(\hat{z}_{kh} - \hat{z}\right)}_{L^\infty(I\times\omega)}\Big)
	\norm{\bar{q}-\bar{q}_{kh}}_{L^1(I\times\omega)}.
	\end{multline*}
	Last, dividing by $\norm{\bar{q}-\bar{q}_{kh}}_{L^1(I\times\omega)}$ yields the desired estimate.	
\end{proof}


\subsection{Concrete control discretization schemes}	
Before we apply the general results of the preceding subsections,
we will verify the equivalence of a semi-variational and an explicit discretization
of the controls.
To this end, let $\Qh \subseteq \Q$ be a finite dimensional subspace.
In the following we consider for given \(\Qh\) the two choices of the control space \(\Qsigma(0,1)\):
the discrete control space \(\Qsigma(0,1) = \Qkh(0,1)\), where
\begin{equation}
\label{eq:space_control_discrete}
\Qkh(0,1) = \left\{v \in \Q(0,1) \constraintSet v|_{I_m} \in
\mathcal{P}_0({I_m};\Qh), \; m = 1,2,\ldots,M\right\},		
\end{equation}
and the semi-variational control space \(\Qsigma(0,1) = L^2(I; \Qh)\).
Additionally, let $\ProjK$ denote the $L^2$-projection onto 
the piecewise constant functions in time.
The problem~\eqref{PkhAlpha} posed with \(\Qsigma(0,1) = \Qkh(0,1)\) is equivalent
to~\eqref{PkhAlpha} with \(\Qsigma(0,1) = L^2(I; \Qh)\) in the following sense.
\begin{proposition}
	\label{prop:equivalence_variational_discrete}		
	If $(\bar{\nu}_{kh}, \bar{q}_{kh})$
	is an optimal solution to~\eqref{PkhAlpha} with \(\Qsigma(0,1) = \Qkh(0,1)\)
	then $(\bar{\nu}_{kh}, \bar{q}_{kh})$ is also optimal for~\eqref{PkhAlpha} with \(\Qsigma(0,1) = L^2(I; \Qh)\).
	Conversely, if $(\bar{\nu}^{\mathrm{v}}_{kh}, \bar{q}^{\mathrm{v}}_{kh})$
	is an optimal solution to~\eqref{PkhAlpha} with \(\Qsigma(0,1) = L^2(I; \Qh)\),
	then $(\bar{\nu}^{\mathrm{v}}_{kh}, \ProjK\bar{q}^{\mathrm{v}}_{kh})$ is also optimal for~\eqref{PkhAlpha} with \(\Qsigma(0,1) = \Qkh(0,1)\).
\end{proposition}
\begin{proof}
	First, since the variational admissible set
	\(L^2(I;\Qh)\cap\Qad(0,1)\) is larger than the fully discrete one \(\Qkh(0,1)\cap\Qad(0,1)\),
	we immediately obtain \(\bar{\nu}^{\mathrm{v}}_{kh} \leq \bar{\nu}_{kh}\) for the
	optimal times.
	Clearly, \(\ProjK\bar{q}^{\mathrm{v}}_{kh} \in \Qkh(0,1)\cap\Qad(0,1)\) by the
	fact that \(\Pi_k\) can be computed explicitly on every interval \(I_m\) as the interval mean.
	In addition, by the orthogonality-properties of the \(L^2\)-projection \(\Pi_k\)
	and the definition of the state equation~\eqref{eq:stateEquationDiscrete},
	$(\bar{\nu}^{\mathrm{v}}_{kh}, \ProjK\bar{q}^{\mathrm{v}}_{kh})$ has the same
	associated discrete state as \((\bar{\nu}^{\mathrm{v}}_{kh},
	\bar{q}^{\mathrm{v}}_{kh})\), which directly implies that
	\(g_{kh}(\bar{\nu}^{\mathrm{v}}_{kh}, \ProjK\bar{q}^{\mathrm{v}}_{kh}) \leq 0\). Thus, $(\bar{\nu}^{\mathrm{v}}_{kh}, \ProjK\bar{q}^{\mathrm{v}}_{kh})$
	is feasible for~\eqref{PkhAlpha} with \(\Qsigma(0,1) = \Qkh(0,1)\) and 
	therefore $\bar{\nu}_{kh} \leq \bar{\nu}^{\mathrm{v}}_{kh}$.
	Hence, both problems have the same optimal time \(\bar{\nu}_{kh} = \bar{\nu}^{\mathrm{v}}_{kh}\).
	Consequently, the optimal controls of both problems are given by all controls
	\(q \in \Qsigmaad(0,1)\) such that \(g_{kh}(\bar{\nu}_{kh}, q) \leq 0\), with
	\(\Qsigma(0,1) = \Qkh(0,1)\) or \(\Qsigma(0,1) = L^2(I;\Qh)\), respectively. A similar
	argument as before yields the relation between the optimal controls as claimed.
\end{proof}
As we are interested in explicit rates of convergence,
for the following considerations we assume that 
$\Psi(\varepsilon) = C\varepsilon^{\kappa}$ in \eqref{eq:assumption_structure_adjoint_ssc_bb}.
The proceeding results hold for a general function $\Psi$ 
satisfying \eqref{eq:assumption_structure_adjoint_ssc_bb}
with obvious modifications.	

\subsubsection{Purely time-dependent controls}
In case of purely time-dependent controls we immediately derive an error estimate (that is optimal if $\kappa = 1$) using the $L^\infty(I; L^2)$ discretization error estimate for the variational control discretization. Note that besides theoretical advantages purely time-dependent controls are also interesting in practice as distributed controls are typically difficult to implement.
\begin{theorem}[Parameter control, variational]
	\label{thm:robust_estimate_bb_purely_timedep}
	Adopt the assumptions of \cref{lemma:robust_error_estimate_bb_suboptimal} and let \eqref{eq:assumption_structure_adjoint_ssc_bb} hold with $\Psi(\varepsilon) = C\varepsilon^\kappa$.
	Additionally, suppose purely time-dependent controls and let
	\((\bar{\nu}_{kh},\bar{q}_{kh})\) be an optimal solution
	of~\eqref{PkhAlpha} with \(\Qsigma(0,1) = L^2(I,\R^{\parameterControlDim})\). Then there is a constant $c > 0$ such that
	\[
	\abs{\bar{\nu}-\bar{\nu}_{kh}} + \norm{\bar{q}-\bar{q}_{kh}}^{1/\kappa}_{L^1(I\times\omega)} \leq c\abs{\log k}(k + h^2).
	\]
\end{theorem}
\begin{proof}
	This follows from \cref{prop:robust_estimate_bb}, since in case of purely time-dependent control we may use the $L^\infty(I; L^2)$ discretization error estimate, see \cref{lemma:errorEstimatesStateEquationLinf}, for the state and adjoint state equation to obtain
	\begin{align*}
	\norm{\ControlOp^*\left(\hat{z}_{kh} - \hat{z}\right)}_{L^\infty(I\times\omega)}
	&= \esssup_{t \in I}\max_{i \in \set{1,\ldots,\parameterControlDim}}\abs{\inner{e_i, \hat{z}_{kh}(t) - \hat{z}(t)}}\\
	&\leq c\norm{\hat{z}_{kh} - \hat{z}}_{L^\infty(I; L^2)}
	\leq c\abs{\log k}(k + h^2).
	\end{align*}
	In addition, $\sigma(k,h) = 0$ as we do not explicitly discretize the control variable.
	The remaining estimate for $\bar{\nu}$ is proved in \cref{lemma:robust_error_estimate_times}.
\end{proof}
Next, we consider an explicitly discretized control variable.
Using \cref{prop:equivalence_variational_discrete}
with $\Qh = \R^{\parameterControlDim}$, we immediately obtain the following result.

\begin{corollary}[Parameter control, discrete]
	\label{thm:robust_estimate_bb_purely_timedep_discrete}
	The result of \cref{thm:robust_estimate_bb_purely_timedep} remains valid for
	piecewise constant controls \(\Qsigma(0,1) = \Qkh(0,1)\) with $\Qh =
	\R^{\parameterControlDim}$ under the same assumptions.
\end{corollary}

\subsubsection{Distributed control with variational control discretization}
Next, we discuss the case of a distributed control, i.e.\ $\omega \subset \Omega$. 
In order to apply \cref{prop:robust_estimate_bb} 
we require pointwise error estimates 
for the adjoint state equation. 
For simplicity, we only consider the particular case that
the control domain $\omega$ has a strict distance to the boundary $\partial\Omega$
of the spatial domain and smooth initial and desired states.
Moreover, assume in the remaining part of this section that the spatial mesh is quasi-uniform.
Based on pointwise best approximations results from \cite{Leykekhman2016a}
we can obtain the following error estimate. 
For its proof we refer to \cite[Sections~5.5.3, 5.5.4]{Bonifacius2018}.
\begin{proposition}[{\cite[Proposition~5.41]{Bonifacius2018}}]
	\label{prop:robust_estimate_bb_bestapproximation_estimate_adjoints}
	Let $\overline{\omega} \subset \Omega$.
	Suppose that $u_0, u_d \in \dom{L^\infty}{-\Lap}$.
	Then there exists a constant $c > 0$, independent of $k$, $h$, $\hat{z}_{kh}$, and
	$\hat{z}$, such that
	\[
	\norm{\ControlOp^*\left(\hat{z}_{kh} - \hat{z}\right)}_{L^\infty(I\times\omega)}
	\leq c\abs{\log k}^4\abs{\log h}^7(k + h^{2}).
	\]
\end{proposition}
We directly infer the following error estimate for the variational control discretization.
\begin{theorem}[Variational discretization]
	\label{thm:robust_estimate_bb_variational}
	Adopt the assumptions of \cref{lemma:robust_error_estimate_bb_suboptimal} and let \eqref{eq:assumption_structure_adjoint_ssc_bb} hold with $\Psi(\varepsilon) = C\varepsilon^\kappa$.
	Moreover, suppose the variational control discretization, i.e.\ $\Qsigma(0,1) = \Q(0,1)$.
	In addition, assume $\overline{\omega} \subset \Omega$ as well as $u_0, u_d \in \dom{L^\infty}{-\Lap}$.
	Then there is a constant $c > 0$, independent of $k$, $h$, $\bar{\nu}_{kh}$, and $\bar{q}_{kh}$, such that
	\[
	\abs{\bar{\nu}-\bar{\nu}_{kh}} + \norm{\bar{q}-\bar{q}_{kh}}^{1/\kappa}_{L^1(I\times\omega)} 
	\leq c\abs{\log k}^4\abs{\log h}^7(k + h^{2}).
	\]
\end{theorem}
\begin{proof}
	This result follows from \cref{lemma:robust_error_estimate_bb_suboptimal,prop:robust_estimate_bb,prop:robust_estimate_bb_bestapproximation_estimate_adjoints},
	since for the variational control discretization we have $\ProjDiscControl = \Id$ and $\sigma(k,h) = 0$.
\end{proof}

\subsubsection{Distributed control with cellwise constant control discretization}

Last, we consider the discretization of the control by cellwise constant functions in space.
Recall that $\sigma(k, h)$ denotes the projection error onto $\Qsigma(0,1)$ measured $L^2(I; H^{-1})$; 
see \eqref{eq:estimate_projection_discrete_controls_bb2}.
Since the control variable has a bang-bang structure,
we cannot expect order $k$ of convergence in $L^2$ in time. 
We therefore first consider a semi-variational control discretization
and obtain the fully discrete result using \cref{prop:equivalence_variational_discrete}. 
Let the discrete space of controls be defined as follows
\begin{equation*}
Q_h = \left\{v \in L^2(\omega) \constraintSet v|_{K} \in \mathcal P_0(K)\;\text{for all }K \in \mathcal{T}^\omega_h\right\},\quad
\Qsigma(0,1) = L^2(I; Q_h)\mbox{.}
\end{equation*}
Hence, the controls are explicitly discretized in space but not explicitly discretized in time,
which is equivalent to the discretization by piecewise and cellwise constant functions.
Let $\ProjHconst$ denote the $L^2(\omega)$-projection onto the cellwise constant functions.
Moreover, for almost every $t \in [0,1]$ we set
\begin{equation*}
\mathcal{S}_{h,t} \ldef \mathcal{T}^\omega_h\setminus\{K \in \mathcal{T}^\omega_h \constraintSet 
\bar{q}(t)|_{K} \equiv q_a \text{~or~} \bar{q}(t)|_K \equiv q_b\}.
\end{equation*}
We first establish error estimates for
$\sigma(k,h)$ with $\ProjDiscControl = \ProjHconst$.

\begin{proposition}\label{prop:robust_estimate_bb_cellwise_constant_sigmas}
	Suppose there are functions $\delta_h \in L^1(I)$, $h > 0$, and a constant $c > 0$ such that
	\begin{equation}\label{eq:robust_estimate_bb_cellwiseconstant_assumption}
	\sum_{K \in \mathcal{S}_{h,t}} \abs{K} \leq \delta_h(t), \quad \text{a.e.~} t\in [0,1],\quad h > 0,
	\end{equation}
	and $\norm{\delta_h}_{L^1(I)} \leq ch$ for all $h > 0$.
	Then the estimate
	\begin{equation}
	\norm{\ControlOp\left(\ProjHconst\bar{q} - \bar{q}\right)}_{L^2(I;H^{-1})} \leq ch^{3/2},\label{eq:robust_estimate_bb_cellwiseconstant_L2Hdual}
	\end{equation}
	holds with a constant $c > 0$ not depending on $h$.
\end{proposition}
\begin{proof}		
	Since $\ProjHconst$ is a projection, for any $v \in H^1$ and $K \in \mathcal{T}_h^\omega$ we have
	\begin{equation*}
	\inner{\ProjHconst\bar{q}(t) - \bar{q}(t), v}_{L^2(K)} 
	\leq ch \norm{\ProjHconst\bar{q}(t) - \bar{q}(t)}_{L^2(K)} \norm{\nabla v}_{L^2(K)}.
	\end{equation*}
	Using H\"older's inequality
	and the supposition~\eqref{eq:robust_estimate_bb_cellwiseconstant_assumption} 
	yields~\eqref{eq:robust_estimate_bb_cellwiseconstant_L2Hdual}.
\end{proof}
We have
the following sufficient condition
for~\eqref{eq:robust_estimate_bb_cellwiseconstant_assumption}, which is proved along the
lines of the proof of \cite[Theorem~4.4]{Casas2017a}.
\begin{proposition}\label{prop:robust_estimate_bb_cellwiseconstant_assumption_sufficient1}
	If $\ControlOp^*\bar{z} \in L^1(I; C^1(\overline{\omega}))$ and \eqref{eq:assumption_structure_adjoint_ssc_bb} holds with $\Psi(\varepsilon) = C\varepsilon$, then
	\eqref{eq:robust_estimate_bb_cellwiseconstant_assumption} is valid.
\end{proposition}
Finally, we provide error estimates for cellwise constant control discretization.
\begin{theorem}[Cellwise constant controls]
	\label{thm:robust_estimate_bb_cellwise_constant}
	Adopt the assumptions of \cref{lemma:robust_error_estimate_bb_suboptimal} and let \eqref{eq:assumption_structure_adjoint_ssc_bb} hold with $\Psi(\varepsilon) = C\varepsilon^\kappa$.
	Moreover, suppose the variational in time and cellwise constant control discretization in space, i.e.\ $\Qsigma(0,1) = L^2(I; Q_h)$. 
	In addition, assume $\overline{\omega} \subset \Omega$, $u_0, u_d \in \dom{L^\infty}{-\Lap}$,
	and that \eqref{eq:robust_estimate_bb_cellwiseconstant_assumption} is satisfied.
	There is a $c > 0$ not depending on $k$, $h$, $\bar{\nu}_{kh}$, and $\bar{q}_{kh}$ such that
	\begin{align*}
	\abs{\bar{\nu}-\bar{\nu}_{kh}}
	&\leq c\abs{\log k}(k + h^{3/2}),\\
	\norm{\bar{q}-\bar{q}_{kh}}^{1/\kappa}_{L^1(I\times\omega)} 
	&\leq c\abs{\log k}^4\abs{\log h}^7(k + h).
	\end{align*}
\end{theorem}	
\begin{proof}
	The error estimate \cref{prop:robust_estimate_bb_bestapproximation_estimate_adjoints}
	and stability of $\Id-\ProjHconst$ in $L^\infty$ yield
	\[
	\norm{(\Id-\ProjHconst)\ControlOp^*\hat{z}_{kh}}_{L^\infty(I\times\omega)}
	\leq c\abs{\log k}^4\abs{\log h}^7(k + h^2) + 
	\norm{(\Id-\ProjHconst)\ControlOp^*\hat{z}}_{L^\infty(I\times\omega)}
	\]
	Moreover, employing elliptic regularity with some $p > d$, we have the estimate
	\begin{equation*}
	\norm{(\Id-\ProjHconst)\ControlOp^*\hat{z}}_{L^\infty(I\times\omega)} 
	\leq ch\norm{\hat{z}}_{L^\infty(I; W^{2,p}(\omega))}
	\leq ch\norm{\hat{z}}_{L^\infty(I; \dom{L^p}{-\Lap})} \leq ch.
	\end{equation*}
	Hence, using \cref{lemma:robust_error_estimate_bb_suboptimal}, \cref{prop:robust_estimate_bb,prop:robust_estimate_bb_bestapproximation_estimate_adjoints} 
	as well as the estimates for $\sigma$ from \cref{prop:robust_estimate_bb_cellwise_constant_sigmas}
	we infer the desired estimate.
\end{proof}

Using \cref{prop:equivalence_variational_discrete}, we immediately obtain the following result.

\begin{corollary}[Piecewise and cellwise constant controls]
	\label{thm:robust_estimate_bb_piecewise_cellwise_constant}
	The result of \cref{thm:robust_estimate_bb_cellwise_constant} remains valid for
	\(\Qsigma(0,1) = \Qkh(0,1)\) under the same assumptions.
\end{corollary}

\section{Numerical examples}
\label{sec:examples_bb}

\newlength\convergencePlotWidth
\setlength\convergencePlotWidth{4.5cm}
\newlength\convergencePlotHeight
\setlength\convergencePlotHeight{4.1cm}
\newlength\timeseriesSize
\setlength\timeseriesSize{2.75cm}

We verify the theoretical results by numerical examples.
In order to solve the optimization problem~\eqref{Pt}, we employ the equivalence of time-
and distance optimal control problems summarized in \cref{sec:algorithm}
(see also~\cite{Bonifacius2018a}), and solve a sequence of
optimization problems with a fixed time. The resulting convex sub-problems for a fixed time are
solved by an accelerated conditional gradient method. In an outer loop the optimal time is
determined by a Newton method.
For further details we refer to \cite{Bonifacius2018a}.
The computations are performed in \textsc{MATLAB}.


\subsection{Example with purely time-dependent control}\label{sec:example2_bb}
We take the example from \cite[Section~5.2]{Bonifacius2017a} with purely time-dependent controls for fixed spatially dependent functions
but without control costs in the objective functional.
Let
\begin{align*}
\Omega &= (0,1)^2\mbox{,}\quad
\omega_1 = (0, 0.5)\times(0,1),\quad \omega_2 = (0.5, 1)\times(0,0.5)\mbox{,}\\
\ControlOp &\colon \R^2 \to L^2(\Omega),\quad \ControlOp q = q_1 \chi_{\omega_1} + q_2 \chi_{\omega_2},\\
\Qad(0,1) &= \{q \in L^2(I; \R^2) \constraintSet  -1.5 \leq q \leq 0\}\mbox{,}\\
u_0(x) &= 4\sin(\pi x_1^2)\sin(\pi x_2^3),\quad u_d(x) = 0, \quad \delta_0 = {1}/{10}\mbox{,}
\end{align*}
where $\chi_{\omega_1}$ and $\chi_{\omega_2}$ denote the characteristic functions on $\omega_1$ and $\omega_2$.	
The spatial mesh is chosen such that the boundaries of $\omega_1$ and $\omega_2$ coincide with edges of the mesh. 
We discretize the control by piecewise constant functions in time.

\begin{figure}[!ht]		
	\begin{center}
		\includegraphics{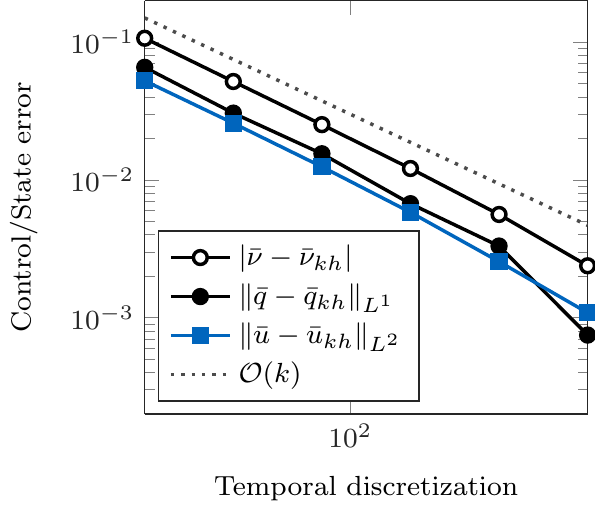}
		\includegraphics{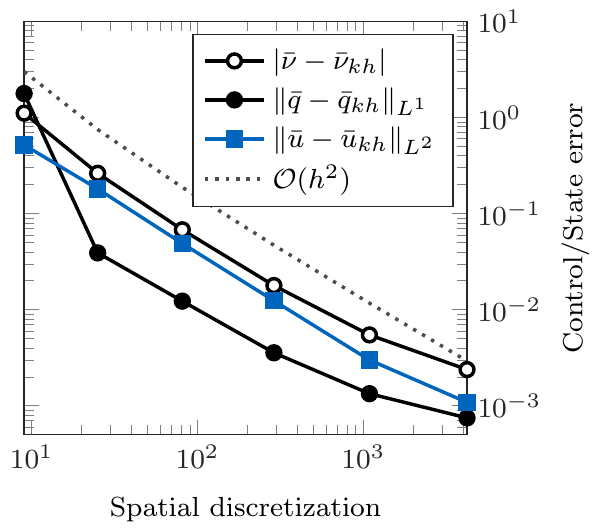}
		\vspace{-.5em}
	\end{center}
	\caption{Discretization error for Example~\ref{sec:example2_bb} with piecewise constant control discretization 
		and refinement of the time interval for $N = 4225$ nodes (left) and refinement of the 
		spatial discretization for $M = 640$ time steps (right). The reference solution is calculated for $N = 16641$ and $M = 1280$.}
	\label{fig:example2_convergence_bb}
\end{figure}

\begin{figure}[!ht]		
	\begin{center}
		\includegraphics{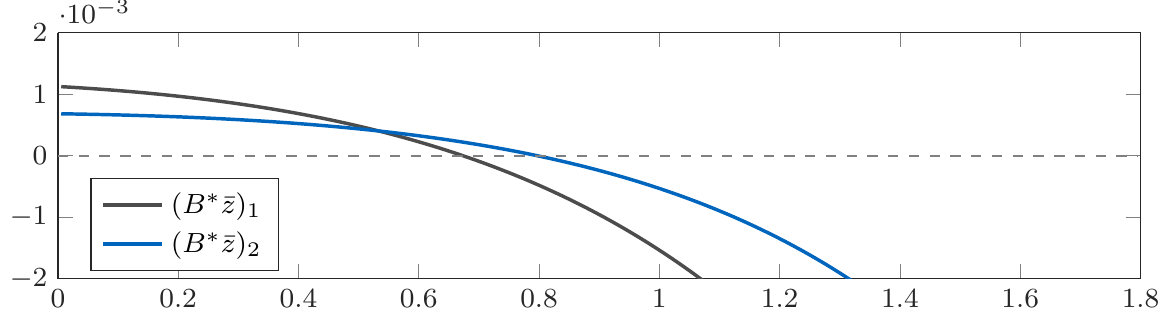}
	\end{center}
	\caption{The switching function $\ControlOp^*\bar{z}$ from
		Example~\ref{sec:example2_bb} near zero.}
	\label{fig:example2_Badjoint}
\end{figure}

Since the exact solution is unknown, we calculate a numerical solution on a sufficiently fine grid. 
In accordance with \cref{thm:robust_estimate_bb_purely_timedep_discrete} 
(provided that~\eqref{eq:assumption_structure_adjoint_ssc_bb} holds with $\Psi(\varepsilon) = C\varepsilon$,
see also the plot of the switching function in \cref{fig:example2_Badjoint} and
the numerical test in \cref{fig:example23_struct_adjoint}), 
we observe linear convergence with respect to $k$
and quadratic order of convergence in $h$ for all variables; see \cref{fig:example2_convergence_bb}.

\subsection{Example with distributed control on subdomain}\label{sec:example3_bb}
Next, we consider the example from \cite[Section~5.3]{Bonifacius2017a} 
with distributed control on the subset \(\omega = (0, 0.75)^2\) of the domain \(\Omega = (0,1)^2\).
As before we compare to a reference solution obtained numerically on a fine grid.
The control bounds are \(q_a = -5\), \(q_b = 0\), and the data is
\begin{align*}
u_d(x) &= -2\min\set{x_1, 1-x_1, x_2, 1-x_2}\mbox{,} \quad \delta_0 = {1}/{10},\\
u_0(x) &= 4\sin(\pi x_1^2) \sin(\pi x_2)^3\mbox{.}
\end{align*}
We consider the piecewise and cellwise constant discretization for the control variable.
As in the first example we observe full order of convergence with respect to the terminal time. 
However, we do not have full order convergence for the control variable. From \cref{fig:example3_convergence_bb} 
we approximately estimate the rate $k^{1/2}$ and $h$, respectively, for the control variable.
\begin{figure}[!ht]
	\begin{center}
		\includegraphics{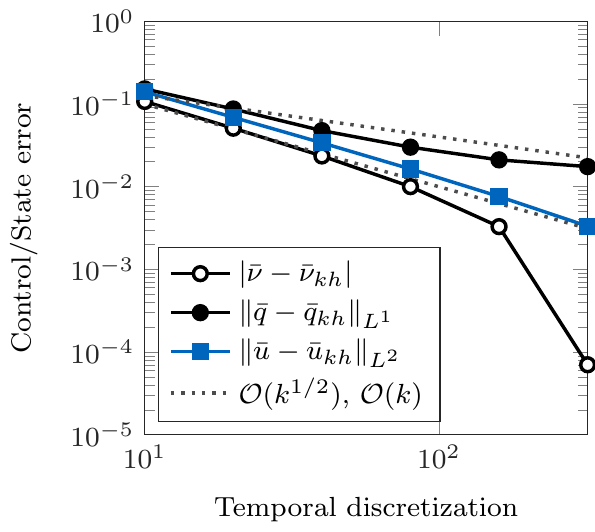}
		\includegraphics{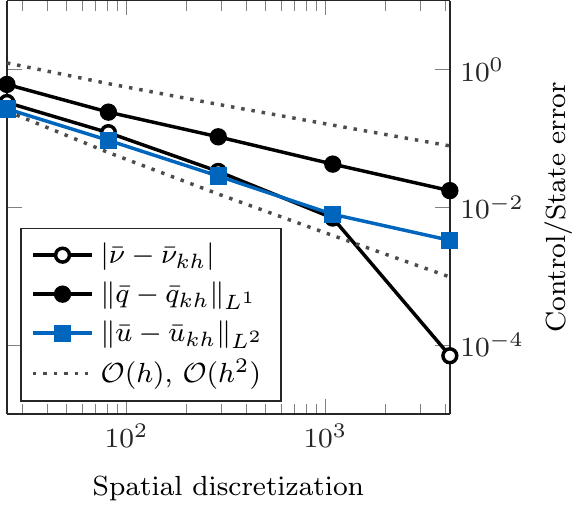}
		\vspace{-.5em}
	\end{center}
	\caption{Discretization error for Example~\ref{sec:example3_bb} with piecewise and cellwise constant 
		control discretization and refinement of the time interval for $N = 4225$ nodes (left) and 
		refinement of the spatial discretization for $M = 320$ time steps (right). The reference solution is calculated for $N = 16641$ and $M = 640$.}
	\label{fig:example3_convergence_bb}
\end{figure}
Numerically evaluating the condition~\eqref{eq:assumption_structure_adjoint_ssc_bb}
we observe that the structural assumption is not satisfied with $\kappa = 1$ in this
example; see \cref{fig:example23_struct_adjoint} (cf.\ also \cref{fig:example3_Badjoint}
for a plot of the switching function).
For this reason, we cannot expect the rate $k$ for the control variable employing \cref{thm:robust_estimate_bb_cellwise_constant}. 
In Example~\ref{sec:example2_bb} we observe linear decrease while in Example~\ref{sec:example3_bb} it is hard to determine the rate of decrease; see \cref{fig:example23_struct_adjoint}.

\begin{figure}[!ht]		
	\begin{center}
		\includegraphics{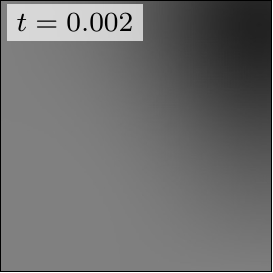}
		\includegraphics{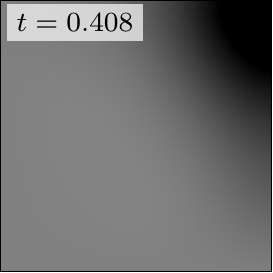}
		\includegraphics{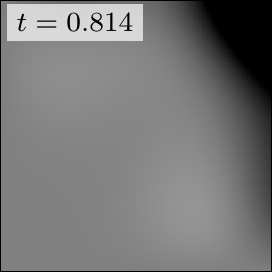}
		\includegraphics{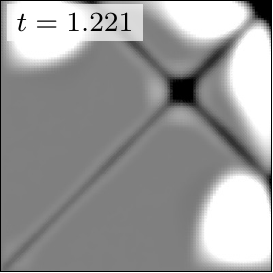}\includegraphics{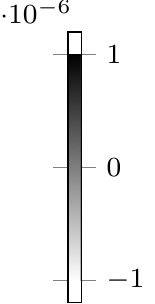}
	\end{center}
	\caption{Snapshots of switching function $\ControlOp^*\bar{z}$ from
		Example~\ref{sec:example3_bb} (with color scale adapted to values below \(10^{-6}\)).}
	\label{fig:example3_Badjoint}
\end{figure}

\begin{figure}[!ht]		
	\begin{center}
		\includegraphics{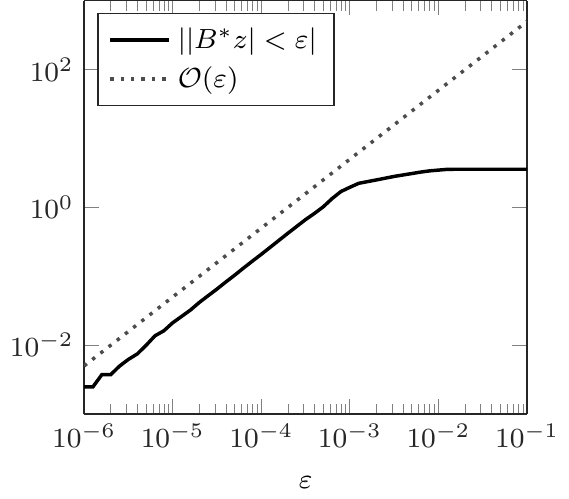}
		\includegraphics{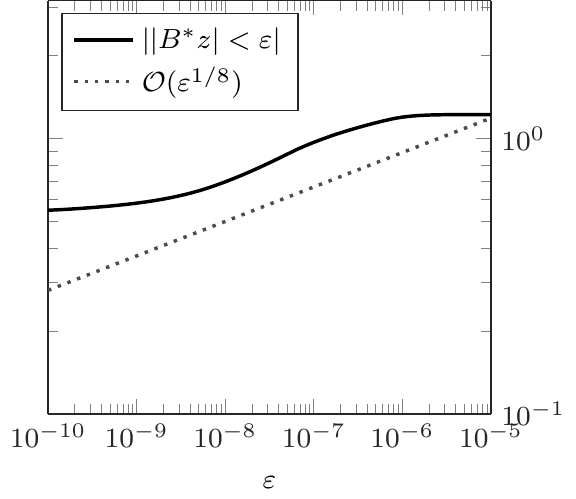}
		\vspace{-.5em}
	\end{center}
	\caption{Numerical verification of structural assumption on adjoint state \eqref{eq:assumption_structure_adjoint_ssc_bb} for Example~\ref{sec:example2_bb} (left) and Example~\ref{sec:example3_bb} (right).}
	\label{fig:example23_struct_adjoint}
\end{figure}

\appendix
\section{Regularity results and discretization error estimates}

\begin{proposition}[{\cite[Proposition~A.19]{Bonifacius2018}}]
	\label{prop:complete_continuity_control_to_terminal_obs}		
	Let $s > 2$ and $u_0 \in L^2$. The mapping $(\nu,q) \mapsto S(\nu,q)$ is completely continuous from $\R\times L^s(I\times\omega)$ into $C([0,1]; L^2)$.
\end{proposition}	
For the proof of \cref{prop:robust_estimate_bb} we require the following Lipschitz-type estimate of the solution to the state equation with respect to the time transformation.
\begin{proposition}\label{prop:estimate_transformations_LinfL2}
	Let $\nu_{\max} > \nu_{\min} > 0$. 
	There is $c > 0$ such that for any $u_0 \in L^2$, 
	$f \in L^2(I;H^{-1})$, and $\nu_1, \nu_2 \in [\nu_{\min},\nu_{\max}]$ the 
	solutions to the state equation
	$u(\nu_1) = u(\nu_1,u_0,f)$ and $u(\nu_2) = u(\nu_2,u_0,f)$ satisfy the estimate
	\[
	\norm{u(\nu_1) - u(\nu_2)}_{C([0,1]; L^2)} \leq 
	c\abs{\nu_1-\nu_2}\left(\norm{f}_{L^2(I; H^{-1})} + \norm{u_0}_{L^2}\right),
	\]
	where $c > 0$ is independent of $\nu$, $f$, and $u_0$.
\end{proposition}
\begin{proof}
	Set $u_1 = u(\nu_1)$ and $u_2 = u(\nu_2)$. 
	Then the difference $w = u_1 - u_2$ satisfies
	\[
	\partial_t w - \nu_1\Lap w = (\nu_1-\nu_2)\left(\Lap u_2 + f\right), \quad w(0) = 0.
	\]
	Hence, standard energy estimates lead to
	\begin{align*}
	\norm{w}_{\MPRHilbert{I}{H^{-1}}{H^1}}
	&\leq c\abs{\nu_1-\nu_2}\norm{-\Lap u_2 + f}_{L^2(I; H^{-1})}\\
	&\leq c\abs{\nu_1-\nu_2}\left(\norm{f}_{L^2(I; H^{-1})} + \norm{u_0}_{L^2}\right).
	\end{align*}
	Last, the assertion follows from
	$\MPRHilbert{I}{H^{-1}}{H^1} \embedding C([0,1]; L^2)$.
\end{proof}
\begin{proposition}\label{prop:estimate_transformations_LinfLinf}
	Let $\nu_{\max} > \nu_{\min} > 0$ and $s, p \in (1,\infty)$
	such that $d/(2p) + 1/s < 1$. 
	There is $c > 0$ such that for any $u_0 \in (L^p,\dom{L^p}{-\Lap})_{1-1/s,s}$, 
	$f \in L^s(I;L^p)$, and $\nu_1, \nu_2 \in [\nu_{\min},\nu_{\max}]$ the 
	solutions to the state equation
	$u(\nu_1) = u(\nu_1,u_0,f)$ and $u(\nu_2) = u(\nu_2,u_0,f)$ satisfy the estimate
	\[
	\norm{u(\nu_1) - u(\nu_2)}_{L^\infty(I\times\Omega)} \leq 
	c\abs{\nu_1-\nu_2}\left(\norm{f}_{L^s(I; L^p)} + \norm{u_0}_{(L^p,\dom{L^p}{-\Lap})_{1-1/s,s}}\right),
	\]
	where $c > 0$ is independent of $\nu$, $f$, and $u_0$.
\end{proposition}
\begin{proof}
	Maximal parabolic regularity of $-\Lap$ on $L^p$,
	see, e.g., \cite[Theorem~2.9~b)]{Disser2015}, yields that
	the solution $u = u(\nu, f, u_0)$ satisfies the estimate
	\[
	\norm{u}_{\MPRSpace{s}{L^p}{\dom{L^p}{-\Lap}}} \leq c\left(\norm{f}_{L^s(I; L^p)} + \norm{u_0}_{(L^p,\dom{L^p}{-\Lap})_{1-1/s,s}}\right).
	\]
	Moreover, continuity of $\nu \mapsto (\partial_t - \nu\Lap)^{-1}$, $\nu > 0$, 
	as well as compactness of $[\nu_{\min},\nu_{\max}]$ imply
	that the constant in the estimate above can be chosen
	uniformly with respect to $\nu$.
	Set $u_1 = u(\nu_1)$ and $u_2 = u(\nu_2)$. 
	Then the difference $w = u_1 - u_2$ satisfies
	\[
	\partial_t w - \nu_1\Lap w = (\nu_1-\nu_2)\left(\Lap u_2 + f\right), \quad w(0) = 0.
	\]
	Hence,
	\begin{align*}
	\norm{w}_{\MPRSpace{s}{L^p}{\dom{L^p}{-\Lap}}}
	&\leq c\abs{\nu_1-\nu_2}\norm{-\Lap u_2 + f}_{L^s(I; L^p)}\\
	&\leq c\abs{\nu_1-\nu_2}\left(\norm{f}_{L^s(I; L^p)} + \norm{u_0}_{(L^p,\dom{L^p}{-\Lap})_{1-1/s,s}}\right).
	\end{align*}
	Finally, the assertion follows from the embedding
	\[
	\MPRSpace{s}{L^p}{\dom{L^p}{-\Lap}} \embedding C(\overline{I\times\Omega});
	\]
	see the proof of \cite[Theorem~3.1]{Disser2015}.
\end{proof}

\begin{lemma}[{\cite[Lemma~B.2]{Bonifacius2017a}}]
	\label{lemma:errorEstimatesStateEquationLinf}		
	Let $\nu \in \Rplus$ and $f \in L^\infty((0,1); L^2)$. For the solution $u = u(\nu,f)$ to the state equation with right-hand side $f$ and the discrete solution $u_{kh} = u_{kh}(\nu,f)$ to equation~\eqref{eq:stateEquationDiscrete} with right-hand side $f$ it holds
	\begin{align}	
	\norm{u - u_{kh}}_{L^\infty(I; L^2)} &\leq c\abs{\log k}\left(k + h^2\right)\left((1 + \nu)\norm{f}_{L^\infty(I; L^2)} + \nu^{-1}\norm{u_0}_{L^2}\right)\label{eq:errorEstimatesStateEquationLinfL2}\mbox{,}\\
	\norm{u - u_{kh}}_{L^\infty(I; L^2)} &\leq c\abs{\log k}\left(k + h^2\right)(1 + \nu)\left(\norm{f}_{L^\infty(I; L^2)} + \norm{\Lap u_0}_{L^2}\right)\label{eq:errorEstimatesStateEquationLinfL2_stable}\mbox{,}
	\end{align}
	where the constant $c$ is independent of $\nu$, $k$, $h$, $f$, $u_0$, and $u$.
\end{lemma}

\begin{proposition}\label{prop:uniformConvergenceStateTeminalValue}
	Let $\nu_{\max} \in \Rplus$, $q \in \Qad(0,1)$, and $u_0 \in L^2$. Then
	\[
	\lim_{k,h \to 0} \sup_{\nu \in (0,\nu_{\max})} \norm{i_1S_{kh}(\nu,q) - i_1S(\nu,q)}_{L^2} = 0.
	\]
\end{proposition}
\begin{proof}
	We abbreviate $u_{kh} = S_{kh}(\nu,q)$ and $u = S(\nu,q)$.
	Consider first the case $q = 0$. Let $\varepsilon > 0$ be given. Due to density of $H^2$ in $L^2$ there exists $u_{0,\varepsilon} \in H^2$ such that $\norm{u_0 - u_{0,\varepsilon}} \leq \varepsilon$. Let $u_\varepsilon$ and $u_{kh,\varepsilon}$ denote the corresponding continuous and discrete solutions to the state equation with initial value $u_{0,\varepsilon}$. Using the stability estimates \cite[Proposition~4.1]{Bonifacius2017a}
	and \cite[Proposition~A.1]{Bonifacius2017a}
	as well as the discretization error estimate~\eqref{eq:errorEstimatesStateEquationLinfL2_stable} we find
	\begin{multline*}
	\norm{u_{kh}(1) - u(1)}_{L^2} \leq \norm{u_{kh}(1) - u_{kh,\varepsilon}(1)}_{L^2} + \norm{u_{kh,\varepsilon}(1) - u_{\varepsilon}(1)}_{L^2} + \norm{u_{\varepsilon}(1) - u(1)}_{L^2}\\
	\leq c \norm{\ProjH\left(u_0 - u_{0,\varepsilon}\right)}_{L^2} + c \abs{\log k}(k + h^2)\norm{\Lap u_{0,\varepsilon}}_{L^2} + c \norm{u_{0,\varepsilon} - u_0}_{L^2}\mbox{,}
	\end{multline*}
	with a constant $c$ independent of $k, h, \nu$, and $\varepsilon$.
	Therefore, employing stability of the projection $\ProjH$ in $L^2$, for $k,h > 0$ sufficiently small such that $\abs{\log k}(k + h^2)\norm{\Lap u_{0,\varepsilon}}_{L^2} \leq \varepsilon$ we obtain the estimate $\norm{u_{kh}(1) - u(1)}_{L^2} \leq c\varepsilon$.	
	In the case $u_0 = 0$, we can directly apply the discretization error estimate~\eqref{eq:errorEstimatesStateEquationLinfL2}.
\end{proof}

\section{Algorithmic aspects}
\label{sec:algorithm}
In order to solve the optimization problem~\eqref{Pt}
one could add a regularization term to the objective functional 
(cf.\ also \cite[Section~5.5]{Bonifacius2018})
and solve the auxiliary problem for a decreasing sequence of 
regularization parameters equipped with a path-following strategy.
However, for small $\alpha$ the resulting problems become computationally very expensive.
In this section we describe an alternative approach based
on an equivalent reformulation.
For further details we refer to \cite{Bonifacius2018a}.

\subsection{Equivalence of time and distance optimal controls}	
For any $\delta \geq 0$ we consider the \emph{perturbed time-optimal control problem}
\begin{equation}\label{TOPT}\tag{\mbox{$P_{\delta}$}}
\begin{aligned}
\mbox{Minimize~} T \quad\mbox{subject to}\quad T&\in\Rplus\mbox{,~} 
q \in \Qad(0,T)\mbox{,}\\
\norm{u_{q}(T) - u_d}_{L^2} &\leq \delta_0 + \delta\mbox{.}
\end{aligned}
\end{equation}
Moreover, for fixed $T > 0$ we consider the \emph{minimal distance control problem}
\begin{equation}\label{DOPT}\tag{\mbox{$P_T$}}
\mbox{Minimize~} \norm{u_{q}(T) - u_d}_{L^2} - \delta_0 \quad\mbox{subject to}\quad q \in \Qad(0,T)\mbox{.}\\
\end{equation}
Note that \eqref{DOPT} is a nonlinear and nonconvex optimization problem subject 
to control as well as state constraints, whereas~\eqref{DOPT} is a convex problem subject 
to control bounds only.

We define the value functions $T \colon [0,\infty) \to [0,\infty]$ and $\delta \colon [0,\infty) \to [0,\infty)$ as
\[
T(\delta) = \inf\eqref{TOPT}\quad\text{and}\quad
\delta(T) = \inf\eqref{DOPT}.
\]
From boundedness of $\Qad$, linearity of the control-to-state mapping (for fixed $T > 0$), 
and weak lower semicontinuity of the norm function, we immediately infer 
that the value function $\delta(\cdot)$ is well-defined. 
Furthermore, under \cref{assumption:existence_feasible_control}, standard arguments lead to well-posedness of $T(\cdot)$.

The problems~\eqref{TOPT} and~\eqref{DOPT} are 
connected to each other in the following sense --
provided that $T(\cdot)$ is left continuous which we 
will assume throughout the remaining article.
If $T \in (0, T(0)]$ and $q \in \Qad(0,T)$ is distance-optimal for~\eqref{DOPT}, 
then $(T,q)$ is also time-optimal for $(P_{\delta(T)})$.
Conversely, if $\delta \in [0, \delta^\bullet]$ 
with $\delta^\bullet = \norm{u_0 - u_d}_{L^2} - \delta_0$ 
and $(T,q) \in \Rplus\times\Qad(0,T)$ is time-optimal for~\eqref{TOPT}, 
then $q$ is also distance-optimal for $(\delta_T)$.

In view of the relation between~\eqref{TOPT} and~\eqref{DOPT},
we are interested in finding a root of 
the value function $\delta(\cdot)$ to solve the 
time-optimal control problem \eqref{P}. 
This leads to a bi-level optimization problem: In the outer loop we 
search for a root of $\delta(\cdot)$ and the inner loop determines for each given $T$ a 
control such that the associated state 
has minimal distance to the target set.

\subsection{Newton method for the outer loop}
Similarly as in \cref{sec:time_optimal_control}, we transform the minimal distance control problem~\eqref{DOPT}
to the reference time interval $(0,1)$.
For fixed $\nu \in \Rplus$,
let $\nu \mapsto \bar{q}(\nu)$ be the (possibly) multi-valued function
\begin{equation}\label{eq:regularized_optimaldistance_opt_control}
\bar{q}(\nu)  = \argmin_{q \in \Qad(0,1)} \; \norm{i_1S(\nu,q)-u_d}_{L^2}.
\end{equation}
We consider the associated value function $\delta \colon \Rplus \to \R$ defined by
\[
\delta(\nu) = \norm{i_1S(\nu,q)-u_d}_{L^2} - \delta_0, \quad q \in \bar{q}(\nu).
\]
Formally differentiating the value function yields
\begin{equation}\label{eq:regularized_optimaldistance_value_function_derivative}
\delta'(\nu) = \int_0^1\pair{\ControlOp q + \Lap u,\bar{z}}\D{t},\quad q \in \bar{q}(\nu),
\end{equation}
where $u = S(\nu,q)$ and $\bar{z} \in W(0,1)$ satisfies
\begin{equation}
\label{eq:optimaldistance_adjoint_state}
-\partial_t \bar{z} - \nu \Lap \bar{z} = 0,\quad 
\bar{z}(1) = \left(\bar{u}(1) - u_d\right)/\norm{\bar{u}(1) - u_d}_{L^2}.
\end{equation}
The resulting Newton method is summarized in \cref{alg:optimaldistance_outer_newton}.
We emphasize that given a solution $q \in \bar{q}(\nu)$, the derivative $\delta'(\nu)$ can be efficiently computed. 
Indeed, the required variables for the evaluation of~\eqref{eq:regularized_optimaldistance_value_function_derivative}
will typically be directly available from the optimization of the inner loop.
For this reason, one step of the Newton method 
has approximately the same computational costs as one step of, e.g., the bisection method.

\begin{algorithm}
	\begin{algorithmic}[1]
		\STATE{Choose $\nu_0 > 0$}
		\FOR{$n = 0, \ldots, n_{\max}$}
		\STATE{Calculate $q_n = \bar{q}(\nu_n)$ using \cref{alg:optimaldistance_inner_gcg} and $u_n = S(\nu_n,q_n)$}
		\IF{$\delta(\nu_n) < \algtol{tol}$}
		\RETURN
		\ENDIF
		
		\STATE{Evaluate $\delta'(\nu_n)$ using \eqref{eq:regularized_optimaldistance_value_function_derivative}}
		\STATE{Set $\nu_{n+1} = \nu_n - \delta(\nu_n) \delta'(\nu_n)^{-1}$}
		\ENDFOR
	\end{algorithmic}
	\caption{Newton method for solution of minimal distance problem}
	\label{alg:optimaldistance_outer_newton}
\end{algorithm}

\subsection{Conditional gradient method for the inner optimization}
For the algorithmic solution of the inner problem, i.e.\ the determination 
of $\bar{q}(\nu)$ in~\eqref{eq:regularized_optimaldistance_opt_control},
we employ the conditional gradient method; see, e.g., \cite{Dunn1980}. 
We abbreviate
\[
f(q) = \norm{i_1S(\nu,q)-u_d}_{L^2}
\]
neglecting the $\nu$ dependence for a moment. 
Differentiability of the control-to-state mapping yields
\[
f'(q)^* = \nu\ControlOp^*z,
\]
where $z \in W(0,1)$ solves \eqref{eq:optimaldistance_adjoint_state} with $u = S(\nu,q)$. 
Given $q_n \in \Qad(0,1)$, we take
\begin{equation}
\label{eq:optimaldistance_gcg_dq}
q_{n+1/2} = \begin{cases}
q_a, &\text{if } \ControlOp^*z_n > 0,\\
q_b, &\text{if } \ControlOp^*z_n < 0,\\
(q_a + q_b)/2, &\text{else},
\end{cases}
\end{equation}
almost everywhere.
The next iterate $q_{n+1}$ is defined by the 
optimal convex combination of $q_n$ and $q_{n+1/2}$, i.e.
\begin{equation}
\label{eq:optimaldistance_gcg_lambda}
\lambda_n = \argmin_{0\leq \lambda \leq 1} f((1-\lambda)q_n + \lambda q_{n+1/2}). 
\end{equation}
This expression can be analytically determined, employing the fact that $q \mapsto S(\nu,q)$ is affine linear.	
Using convexity of $f$ and the definition of $q_{n+1/2}$, we immediately obtain the following a posteriori error estimator
\[
0 \leq f(q_n)-f(\bar{q}) \leq f'(q_n)(q_n - \bar{q}) 
\leq \max_{q \in \Qad(0,1)} f'(q_n)(q_n - q) 
= f'(q_n)(q_n - q_{n+1/2}).
\]
The expression on the right-hand side can be efficiently evaluated using the adjoint representation 
and serves as a termination criterion for the conditional gradient method. 
The algorithm for the inner optimization is summarized in \cref{alg:optimaldistance_inner_gcg}.

\begin{algorithm}	
	\begin{algorithmic}[1]
		\STATE{Let $\nu > 0$ be given. Choose $q_0 \in \Qad(0,1)$}
		\FOR{$n = 0, \ldots, n_{\max}$}
		\STATE{Calculate $u_n = S(\nu,q_n)$ and $z_n$}
		\STATE{Choose $q_{n+1/2}$ as in \eqref{eq:optimaldistance_gcg_dq}}
		\IF{$f'(q_n)(q_n - q_{n+1/2}) < \algtol{tol}$}
		\RETURN
		\ENDIF
		
		\STATE{Calculate $\lambda_n$ as in \eqref{eq:optimaldistance_gcg_lambda}}
		\STATE{Set $q_{n+1} = (1-\lambda_n)q_n + \lambda_n q_{n+1/2}$}
		\ENDFOR
	\end{algorithmic}
	\caption{Conditional gradient method for solution of \eqref{eq:regularized_optimaldistance_opt_control}}
	\label{alg:optimaldistance_inner_gcg}
\end{algorithm}		

Under a structural assumption on the adjoint state such as \eqref{eq:assumption_structure_adjoint_ssc_bb} with $\Psi(\varepsilon) = C\varepsilon$
and purely time-dependent controls the conditional gradient method is known to converge q-linearly; cf.\ \cite[Theorem~3.1~(iii)]{Dunn1980}. 
However, in general only sublinear convergence is guaranteed; see \cite[Theorem~3.1~(i)]{Dunn1980}.
For this reason, we have implemented an acceleration strategy, where 
instead of~\eqref{eq:optimaldistance_gcg_lambda} we use the best convex combination
of all iterates $q_{j+1/2}$, $j = 0,1,2,\ldots,n+1$, with $q_{0+1/2} \ldef q_0$.

\newpage
\bibliographystyle{siamplain}

\end{document}